\documentclass[12pt,english]{amsart}
\usepackage[T1]{fontenc}
\usepackage[latin9]{inputenc}
\usepackage{xcolor}
\usepackage{babel}
\usepackage{amstext}
\usepackage{amsthm}
\usepackage{amssymb}
\usepackage{graphicx}
\usepackage{geometry}
\PassOptionsToPackage{normalem}{ulem}
\usepackage{ulem}
\usepackage[pdfusetitle,
 bookmarks=true,bookmarksnumbered=false,bookmarksopen=false,
 breaklinks=false,pdfborder={0 0 1},backref=false,colorlinks=false,pdfpagemode=FullScreen]
 {hyperref}

\makeatletter

\providecolor{lyxadded}{rgb}{0,0,1}
\providecolor{lyxdeleted}{rgb}{1,0,0}
\DeclareRobustCommand{\mklyxadded}[1]{\textcolor{lyxadded}\bgroup#1\egroup}
\DeclareRobustCommand{\mklyxdeleted}[1]{\textcolor{lyxdeleted}\bgroup\mklyxsout{#1}\egroup}
\DeclareRobustCommand{\mklyxsout}[1]{\ifx\\#1\else\sout{#1}\fi}

\numberwithin{equation}{section}
\numberwithin{figure}{section}
\theoremstyle{plain}
\newtheorem{thm}{\protect\theoremname}[section]
\theoremstyle{definition}
\newtheorem{defn}[thm]{\protect\definitionname}
\theoremstyle{definition}
\newtheorem{example}[thm]{\protect\examplename}
\theoremstyle{remark}
\newtheorem*{rem*}{\protect\remarkname}
\theoremstyle{plain}
\newtheorem{prop}[thm]{\protect\propositionname}
\theoremstyle{plain}
\newtheorem{cor}[thm]{\protect\corollaryname}
\theoremstyle{plain}
\newtheorem{lem}[thm]{\protect\lemmaname}

\usepackage{babel}
\usepackage{color}

\providecommand{\definitionname}{Definition}
\providecommand{\examplename}{Example}
\providecommand{\lemmaname}{Lemma}
\providecommand{\propositionname}{Proposition}
\providecommand{\remarkname}{Remark}
\providecommand{\theoremname}{Theorem}
\definecolor{green}{RGB}{0, 180, 0}

\definecolor{cyan}{RGB}{0, 180, 180}

\definecolor{yellow}{RGB}{211,211,0}

\usepackage{upgreek}

\newcommand{\comment}[1]{}

\providecommand{\corollaryname}{Corollary}

\makeatother

\providecommand{\corollaryname}{Corollary}
\providecommand{\definitionname}{Definition}
\providecommand{\examplename}{Example}
\providecommand{\lemmaname}{Lemma}
\providecommand{\propositionname}{Proposition}
\providecommand{\remarkname}{Remark}
\providecommand{\theoremname}{Theorem}

\begin{document}

\global\long\def\rmi{\mathbf{\textrm{i}}}%
\global\long\def\rme{\mathbf{\textrm{e}}}%
\global\long\def\rmd{\mathbf{\textrm{d}}}%
\global\long\def\id{\mathbf{1}}%
\global\long\def\d{\partial}%
\global\long\def\Z{\mathbb{Z}}%
\global\long\def\T{\mathbb{T}}%
\global\long\def\R{\mathbb{R}}%
\global\long\def\RR{\mathbb{\overline{R}}}%
\global\long\def\C{\mathbb{C}}%
\global\long\def\N{\mathbb{N}}%
\global\long\def\Q{\mathbb{Q}}%
\global\long\def\B{\mathcal{B}}%
\global\long\def\D{\mathcal{D}}%
\global\long\def\CC{\mathcal{C}}%
\global\long\def\S{\mathfrak{S}_{\Omega}}%
\global\long\def\set#1#2{\left\{  #1~:~#2\right\}  }%
\global\long\def\spec#1{\mathrm{Spec}\left(#1\right)}%
\global\long\def\br#1{\left(#1\right)}%
\global\long\def\bra#1{\left\langle #1\right\rangle }%

\global\long\def\sf#1{\mathrm{SF}_{#1}}%

\global\long\def\E{\mathcal{E}}%
\global\long\def\Ev{\mathcal{E}_{v}}%
\global\long\def\V{\mathcal{V}}%
\global\long\def\lmin{\mathbf{\textrm{\ensuremath{\ell_{\text{min}}}}}}%
\global\long\def\vp{v_{+}}%
\global\long\def\vm{v_{-}}%
\global\long\def\vpm{v_{\pm}}%
\global\long\def\A{\mathcal{A}}%
\global\long\def\va{v_{a}}%
\global\long\def\fe{f_{E}}%
\global\long\def\fwe{f_{\omega,E}}%
\global\long\def\mnw{\mathcal{M}_{n}\left(\omega,E\right)}%
\global\long\def\mw{\mathcal{M}\left(\omega,E\right)}%
\global\long\def\No#1{N_{\Omega}\left(#1\right)}%
\global\long\def\ct#1{\#_{a}^{#1}\left(\omega\right)}%
\global\long\def\freq#1{\nu_{#1}}%
\global\long\def\pa{\left.\Gamma_{\omega}\right|_{\left[0,n\right]}}%
\global\long\def\Ha{\left.H_{\omega}\right|_{\left[0,n\right]}}%
\global\long\def\dpa{\left.G_{\omega}\right|_{\left[0,n\right]}}%

\global\long\def\arccot{\mathrm{arccot}}%


\global\long\def\gr{\Gamma_{\omega,t}}%

\global\long\def\grh{\Gamma_{\omega,t}^{\textrm{horiz}}}%

\global\long\def\gra{\Gamma_{a}}%


\global\long\def\sc{n_{\omega,t}}%

\global\long\def\sch{n_{\omega,t}^{\textrm{horiz}}}%

\global\long\def\sca{n^{\left(a\right)}}%


\global\long\def\zc{Z_{\omega,t}}%

\global\long\def\zch{Z_{\omega,t}^{\textrm{horiz}}}%

\global\long\def\zca{Z^{\left(a\right)}}%

\title[Gap labels for ergodic operators on decorated graphs]{Johnson--Schwartzman Gap Labelling\\
 for Metric and Discrete Decorated Graphs}
\author{Ram Band \and Gilad Sofer}
\address{Department of Mathematics\\
Technion -- Israel Institute of Technology\\
Haifa, Israel}
\email{ramband@technion.ac.il}
\address{Department of Mathematics and the Helen Diller Quantum Center\\
Technion -- Israel Institute of Technology\\
Haifa, Israel}
\email{gilad.sofer@campus.technion.ac.il}
\begin{abstract}
We study Schrödinger operators on metric and discrete decorated graphs.
The values taken by the integrated density of states (IDS) on spectral
gaps are called gap labels. A natural question is which gap labels
can occur. We answer this for graphs arising from uniquely ergodic
one-dimensional dynamical systems by proving Johnson--Schwartzman
gap-labelling theorems in both the metric and discrete settings.

Our results extend Johnson--Schwartzman gap labelling beyond the
standard one-dimensional setting. Unlike in one dimension, these graphs
may contain cycles, which prevent the use of Sturm oscillation theory
and require different spectral methods.

We also analyze discontinuities of the IDS for certain graph families
and show that not every admissible label corresponds to an open spectral
gap. This reveals a mechanism of gap closing driven by graph geometry
rather than by the underlying dynamics.
\end{abstract}

\maketitle

\section{Introduction and main results \label{sec: Intro and main results}}

This paper studies the integrated density of states (IDS) of Schrödinger
operators on discrete and metric graphs constructed through ergodic
one-dimensional dynamical systems. The IDS, which roughly counts the
number of eigenstates per unit volume below a given energy level,
is a widely studied object in spectral theory, quantum mechanics,
and solid state physics. It is an important tool for characterizing
the spectral gaps -- the connected components of the spectrum's complement.

Since the IDS is monotone increasing and is constant at spectral gaps,
each gap can be assigned a specific label based on the value of the
IDS within the gap. The \emph{gap labels} of an operator are of significant
physical importance, for instance for characterizing the Hall conductance
in the Integer Quantum Hall Effect \cite{AvrOsaSei_phystod03}.

Traditionally, the first step in deriving the gap labels is by determining
the set of their allowed values, in the form of a gap labelling theorem
(GLT). Historically, proving gap labelling theorems often involved
using K-theory, as originated in \cite{BelLimTes_cmp83,Bellissard1992}
(see also \cite{Kellendonk2024} for a modern review). Nevertheless,
for one-dimensional ergodic systems, there is an alternative approach
to gap labelling. This was first done by Johnson \cite{Johnson_jde86}
(see \cite{JohMos_cmp82,Schwartzman_anmath57} for additional background),
who showed that for certain Schrödinger operators on $\R$, the IDS
takes values in a countable group that can be explicitly computed
via a homomorphism introduced by Schwartzman. Since then, and especially
in recent years, this approach (known as Johnson--Schwartzman gap
labelling) has been successfully extended to ergodic Schrödinger operators
on $\Z$, Jacobi matrices, and CMV matrices (see \cite{damFil_book22,DamFil_otaa23,Damanik2023,DamLi_jfa2025}
and references therein). It is known that for one-dimensional systems
and whenever both approaches (K-theory and Johnson--Schwartzman)
provide a well-defined label set, these label sets agree \cite{Kellendonk_private_2026}.

While Johnson--Schwartzman gap labelling has been developed for various
one-dimensional systems, many physical systems are modeled by more
complicated network-like structures. This naturally leads to the study
of Schrödinger operators on discrete and metric (quantum) graphs.
These serve as models for various physical systems, and often exhibit
interesting spectral properties which are not found in standard one-dimensional
systems. With this in mind, the goal of this paper is to extend the
Johnson--Schwartzman gap labelling to ergodic Schrödinger operators
on metric and discrete graphs. In contrast to $K$-theoretic methods,
the Johnson--Schwartzman gap-labelling is more directly tied to oscillatory
properties of eigenfunctions, making it particularly suitable for
Schrödinger operators on graphs. The graphs studied here, called decorated
$\Z$-graphs, are inspired by one-dimensional tilings, see Figure
\ref{fig: TilingGraphs}. Here, the ergodic dynamical system determines
the geometry of the graph itself, rather than just the potential.
Developing Johnson--Schwartzman gap labelling for these graphs requires
going beyond classical arguments based on Sturm's oscillation theorem,
since these graphs contain cycles. We do so via analysis of these
graphs' non-trivial nodal count, and with further tools from the spectral
analysis on metric graphs. Finally, we show that for certain non-generic
metric graphs, not all predicted gap labels actually appear as IDS
gap labels, due to jump discontinuities in the IDS. We explicitly
express all the energies of these discontinuities and the corresponding
jump values for the class of Sturmian comb graphs.

\ 

The remainder of the paper is structured as follows: the following
subsections provide the necessary background for stating our main
results, which are then presented in Subsection \ref{subsec:Main-results}.
Section \ref{sec:Gap-labeling} presents additional necessary tools,
followed by a proof of the GLT for metric graphs (Theorem \ref{thm:GLT}).
Section \ref{sec:Discrete-GLT} then presents the proof of the GLT
for discrete graphs (Theorem \ref{thm:Discrete-GLT}). Section \ref{sec:Discontinuities-in-the-IDS}
studies the existence of discontinuities in the IDS for Sturmian comb
graphs. Appendix~\ref{sec:IDS-exists} presents results about the
existence of the IDS for metric decorated graphs (Proposition~\ref{prop:IDS-existence}),
and Appendix~\ref{sec:Proof-of-counting-lemma-1} contains the proof
of Lemma \ref{lem:Counting-lemma}, which is needed for proving the
metric GLT.

\subsection{Discrete and metric graphs\label{subsec:Quantum-graph-preliminaries}}

The discrete graphs in this work are denoted by $G=\left(\mathcal{V},\mathcal{E}\right)$
(with the vertex and edge sets sometimes denoted $\V_{G}$, $\E_{G}$
for emphasis). For a vertex $v\in\V_{G}$, let $\Ev$ denote the set
of edges incident to $v$. The degree of a vertex $v$, denoted $\deg(v)$,
is the number of edges in $\Ev$. The discrete graphs here are equipped
with the \textit{normalized discrete Laplacian} $\Delta$, acting
on $\ell^{2}\left(G\right)$ as
\begin{equation}
\Delta\psi\left(v\right)=\psi\left(v\right)-\sum_{u\in\Ev}\frac{1}{\sqrt{\deg\left(v\right)\deg\left(u\right)}}\psi\left(u\right).\label{eq:norm-lap}
\end{equation}
We consider infinite, connected graphs, with vertex degrees $\deg\left(v\right)$
uniformly bounded from above. With these assumptions, $\Delta$ is
bounded and self-adjoint, and its spectrum $\spec{\Delta}$ is contained
in $\left[0,2\right]$.

A \emph{metric graph} $\Gamma=\left(G,\ell\right)$ consists of a
discrete graph $G$, together with a length function $\ell:\E\rightarrow\mathbb{R}_{+}$
which assigns a positive length $\ell_{e}$ to each edge $e\in\mathcal{E}$.
This equips $\Gamma$ with the natural structure of a metric space,
by identifying each edge $e\in\E$ with the interval $\left[0,\ell_{e}\right]$.

A \emph{quantum graph} is a metric graph $\Gamma$ equipped with a
self-adjoint differential operator $H$ acting on the Sobolev space
$H^{2}\left(\Gamma\right):=\oplus_{e\in\mathcal{E}}H^{2}\left(0,\ell_{e}\right)$.
The most common example is the Schrödinger operator $H=-\frac{d^{2}}{dx^{2}}+q(x)$,
where $q\in L^{\infty}\left(\Gamma\right)$ is real-valued, along
with vertex conditions which render $H$ self-adjoint. The most common
choice for the vertex conditions is known as the \textit{Neumann-Kirchhoff}
conditions (or standard conditions), which require:

1. The function $f$ is continuous at each $v\in\V$, i.e.,
\begin{equation}
\ensuremath{f|_{e}\left(v\right)=f|_{e'}\left(v\right),\forall e,e'\in\Ev}.\label{eq:-15-2}
\end{equation}

2. The sum of the derivatives of $f$ at $v$, taken in the outward
direction along each edge, is zero:

\begin{align}
 & \text{\ensuremath{\sum_{e\in\Ev}f'|_{e}\left(v\right)=0.}}\label{eq:-16-2}
\end{align}

In this work, we assume that the edge lengths $\ell_{e}$ are uniformly
bounded from above and below. Under this condition (together with
the assumptions above regarding the combinatorial structure of $G$),
the associated Neumann-Kirchhoff Laplacian is self-adjoint and non-negative
(see \cite{Berkolaiko_qg-intro17,BerKuc_graphs,GnuSmi_ap06,Kurasov_book24}
for an extensive introduction to quantum graphs).

\subsection{Dynamics\label{subsec: Dynamics}}

The graphs considered here are defined through one-dimensional dynamical
systems, which govern their geometric structure. We now introduce
the relevant definitions, and refer to \cite{Baake2013,DamFill_book_vol_2,damFil_book22,Lothaire2002}
for additional background.

Let $\mathcal{A}$ be a finite set, called an \textit{alphabet}, and
consider the space of bi-infinite sequences $\mathcal{A}^{\mathbb{Z}}$.
We equip this space with the product topology, as induced by the following
metric:
\begin{equation}
d\left(\omega,\omega'\right)=\sum_{n\in\Z}\frac{1-\delta_{\omega\left(n\right),\omega'\left(n\right)}}{2^{\left|n\right|}},\label{eq:disc-metric}
\end{equation}
where $\delta_{i,j}$ is the Kronecker delta. The space $\mathcal{A}^{\mathbb{Z}}$
is naturally equipped with the \textit{shift} map (or left shift):
\begin{align}
 & T:\mathcal{A}^{\mathbb{Z}}\rightarrow\mathcal{A}^{\mathbb{Z}},\label{eq:shift}\\
 & T\omega\left(n\right)=\omega\left(n+1\right).\label{eq:shift2}
\end{align}

\begin{defn}
\label{def: subshift}A \textit{subshift} is a closed, $T$-invariant
subset $\Omega\subset\mathcal{A}^{\mathbb{Z}}$. We say that $\Omega$
is uniquely ergodic if there exists a unique $T$-invariant probability
measure $\mu$ on $\Omega$.
\end{defn}

We define the letter counting function for $a\in\A$ on $\omega\in\Omega$,
by
\begin{equation}
\ct N:=\#\set{n\in\left\{ 0,...,N-1\right\} }{\omega\left(n\right)=a}.\label{eq:counting}
\end{equation}
For a uniquely ergodic subshift $\Omega$, the letter frequencies
\begin{equation}
\freq a=\lim_{N\rightarrow\infty}\frac{\ct N}{N},\label{eq:frequency}
\end{equation}
are well-defined, independent of $\omega\in\Omega$, and satisfy $\sum_{a\in\mathcal{A}}\freq a=1$
(\cite[prop. 4.4]{Baake2013}, \cite{Oxtoby1952}).
\begin{example}
\label{exa:Sturmian}Let $\alpha\in\left(0,1\right)\backslash\mathbb{Q}$
and $\theta\in[0,1)$. A Sturmian sequence $\omega_{\alpha,\theta}$
over the alphabet $\A=\left\{ 0,1\right\} $ is defined by 
\begin{equation}
\omega_{\alpha,\theta}\left(n\right)=\chi_{[1-\alpha,1)}\left(n\alpha+\theta\text{ mod \ensuremath{1}}\right).\label{eq:Sturmian-sequence}
\end{equation}
 The \emph{Sturmian subshift} is then given by
\[
\Omega_{\alpha}:=\overline{\set{\omega_{\alpha,\theta}}{\theta\in[0,1)}}\subset\A^{\Z},
\]
 and is a uniquely ergodic subshift, with letter frequencies $\alpha$
and $1-\alpha$ for $1$ and $0$, respectively, \cite{Queffelec_book10}.
\end{example}

\subsection{Decorated $\protect\Z$-graphs\label{subsec:Comb-graphs}}

We now introduce the class of \textit{decorated $\Z$-graphs}, which
are the graphs whose IDS is analyzed in this paper. To define those
graphs, we fix a uniquely ergodic subshift $\left(\Omega,T\right)$
over a finite alphabet $\mathcal{A}$. We use these dynamics to define
both metric and discrete families of graphs.

\subsubsection{Metric decorated $\protect\Z$-graphs}

To each $a\in\A$, we associate a compact metric graph $\gra$, which
we call a \textit{decoration}; it may consist of just a single vertex.
We also select a distinguished \textit{base vertex} $v_{a}\in\V_{\gra}$
in each decoration. Given $L>0$, we construct a family of infinite
metric graphs $\Gamma_{\Omega}:=\left(\Gamma_{\omega}\right)_{\omega\in\Omega}$
as follows: for each $\omega\in\Omega$, we begin with the bi-infinite
chain graph whose vertices are $L\Z$. To each vertex $Ln\in L\Z$,
attach the graph $\Gamma_{\omega\left(n\right)}$, by identifying
the base vertex $v_{a}\in\V_{\Gamma_{\omega\left(n\right)}}$ with
the vertex $Ln$ (see Figure \ref{fig: TilingGraphs}). This produces
an infinite metric graph $\Gamma_{\omega}$, obtained by decorating
the chain graph $\Z$ with the graphs $\{\gra\}_{a\in\A}$ according
to $\omega\in\Omega$.

We define the normalized length which is assigned to the graph family
$\Gamma_{\Omega}$ by 
\begin{equation}
\overline{L}\left(\Gamma_{\Omega}\right):=L+\sum_{a\in\mathcal{A}}\freq a\ell_{a},\label{eq:norm-length}
\end{equation}
where $L$ is the horizontal distance between consecutive decorations,
$\freq a$ is the frequency of $a\in\A$ (\ref{eq:frequency}), and
$\ell_{a}$ is the total length of the decoration $\gra$. Since the
frequencies $\va$ are independent of $\omega\in\Omega$, the normalized
length (\ref{eq:norm-length}) may also be expressed through the average
growth rate of geodesic balls (which is independent of the choice
of $\omega\in\Omega$):
\begin{equation}
\overline{L}\left(\Gamma_{\Omega}\right)=\lim_{r\rightarrow\infty}\frac{\left|\left.\Gamma_{\omega}\right|_{B\left(x,r\right)}\right|}{2\frac{r}{L}},\quad\quad\forall\omega\in\Omega,\quad x\in\Gamma_{\omega},\label{eq:norm-length-2}
\end{equation}
where $B_{\omega}\left(x,r\right)$ is the geodesic ball of radius
$r$ around $x\in\Gamma_{\omega}$, and $\left|\cdot\right|$ is the
standard Lebesgue measure.

\begin{figure}
\includegraphics[scale=0.5]{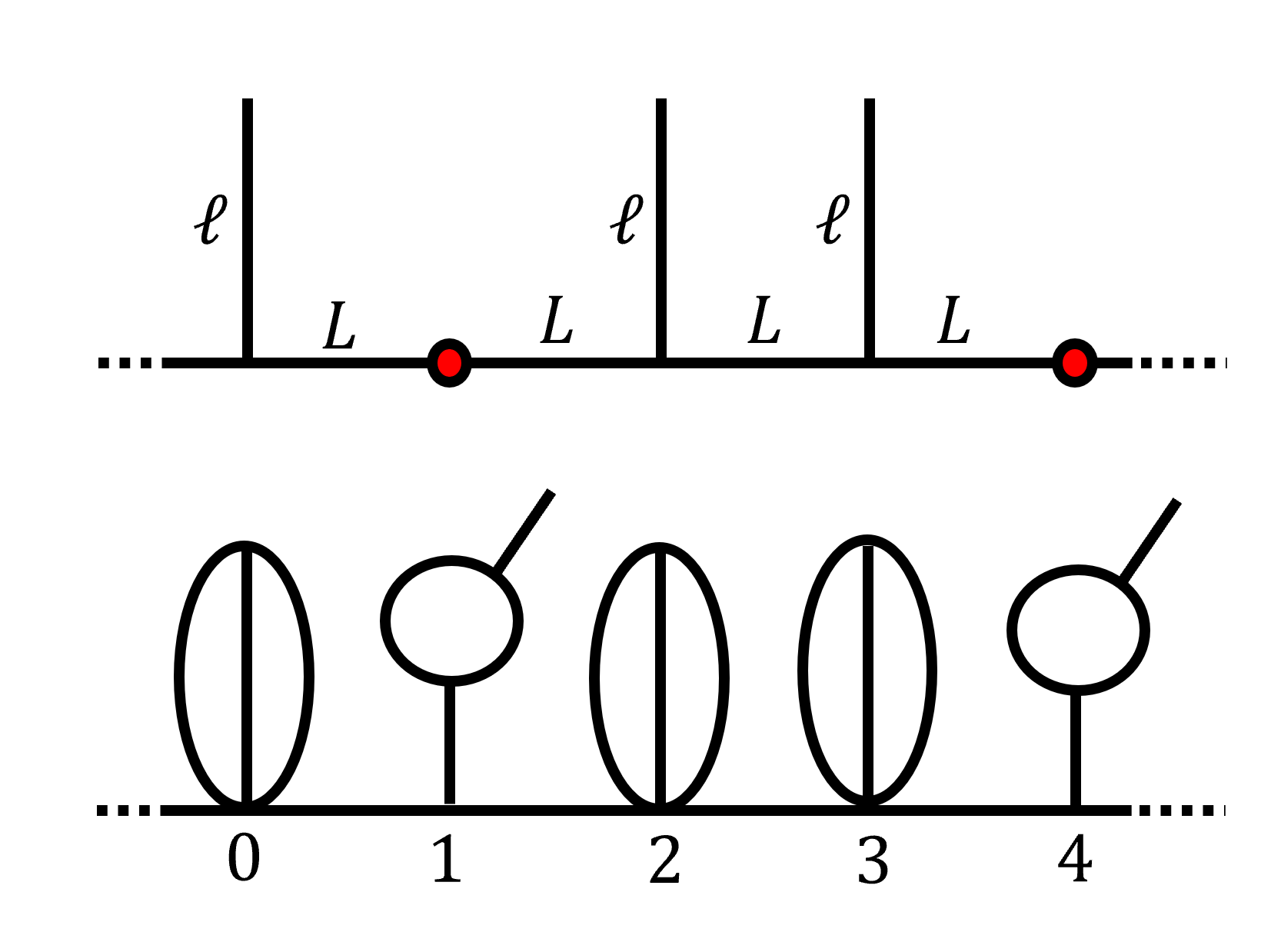}

\caption[Orientation on the edges.]{A Sturmian comb (top), along with a decorated $\protect\Z$-graph
with more complex decorations.\label{fig: TilingGraphs}}
\end{figure}

\begin{example}[Sturmian comb]
\label{exa:Sturmian comb}Taking the Sturmian subshift $\Omega_{\alpha}$
from Example \ref{exa:Sturmian}, one may construct for each $\omega\in\Omega_{\alpha}$
a decorated $\Z$-graph $\Gamma_{\omega}$ by taking the bi-infinite
chain with vertices $L\mathbb{Z}$, and attaching a dangling edge
of length $\ell>0$ at all vertices $Ln$ such that $\omega\left(n\right)=1$
(see Figure \ref{fig: TilingGraphs}). In the notations above, this
means that $\Gamma_{1}$ is a single edge graph of length $\ell$
and $\Gamma_{0}$ is the single vertex graph. In this case,
\begin{equation}
\overline{L}\left(\Gamma_{\Omega}\right)=L+\alpha\ell.\label{eq:L-Sturm}
\end{equation}

~
\end{example}

Any family of graphs $\Gamma_{\Omega}$ is equipped with a naturally
induced shift, 
\begin{align}
 & T:\Gamma_{\Omega}\rightarrow\Gamma_{\Omega},\label{eq:randomgraph1}\\
 & T\Gamma_{\omega}=\Gamma_{T\omega},\label{eq:randomgraph2}
\end{align}
where we abuse the notation $T$. One can further define
\begin{align}
 & T:C\left(\Gamma_{\omega}\right)\rightarrow C\left(\Gamma_{T\omega}\right),\label{eq:Koopman}\\
 & \left(Tf\right)\left(x\right)=f\left(T^{-1}x\right),
\end{align}
where we once again abuse the notation $T$. Equipping each decorated
graph $\Gamma_{\omega}$ with the Kirchhoff Laplacian $H_{\omega}$,
one can consider the family $H_{\Omega}:=\left(H_{\omega}\right)_{\omega\in\Omega}$
as a dynamical system of operators, and get that the family $H_{\Omega}$
is covariant, i.e.,
\begin{equation}
H_{T\omega}=TH_{\omega}T^{-1},\,\forall\omega\in\Omega,\label{eq:covariant}
\end{equation}
which means that the operators $H_{\omega}$ and $H_{T\omega}$ are
unitarily equivalent. Unique ergodicity implies that $\spec{H_{\omega}}$
is in fact almost-surely independent of $\omega\in\Omega$ (see also
\cite{BanSof_prep24a}), and so we can simply denote it by $\spec{H_{\Omega}}$.
\begin{rem*}
Most results in this work should also hold true when the decorated
graphs are equipped with Schrödinger operators whose potentials and
vertex conditions are naturally compatible with the subshift $\left(\Omega,T\right)$.
For simplicity, we focus here on the Kirchhoff Laplacian.
\end{rem*}

\subsubsection{Discrete decorated $\protect\Z$-graphs}

We consider a discrete version of decorated $\Z$-graphs, constructed
in the same manner. Let $\left(\Omega,T\right)$ be a uniquely ergodic
subshift over an alphabet $\A$. Let $\left(G_{a}\right)_{a\in\A}$
be a set of discrete graphs (the possible decorations), each assigned
a base vertex $v_{a}\in\V_{G_{a}}$. Form a family of \textit{discrete
decorated }$\Z$-graph\textit{s} $G_{\Omega}:=\left(G_{\omega}\right)_{\omega\in\Omega}$
as follows: the graph $G_{\omega}$ is constructed from the chain
graph $\Z$ by attaching to each vertex $n\in\Z$ the decoration $G_{\omega(n)}$,
via the identification of the vertex $n\in\Z$ with the base vertex
of $G_{\omega\left(n\right)}$. Each graph $G_{\omega}$ is equipped
with the normalized discrete Laplacian, $\Delta_{\omega}$. We get
the operator family $\Delta_{\Omega}:=\left(\Delta_{\omega}\right)_{\omega\in\Omega}$,
and as above, $\spec{\Delta_{\omega}}$ is almost-surely independent
of $\omega\in\Omega$, and is simply denoted by $\spec{\Delta_{\Omega}}$.

In this setting, the analogue of the normalized length will be the
average number of vertices:
\begin{align}
 & \overline{V}\left(G_{\Omega}\right):=\sum_{a\in\A}\freq a\left|\V_{G_{a}}\right|.\label{eq:discrete-norm-l}
\end{align}

\subsection{Integrated density of states (IDS)\label{subsec:IDS}}

\subsubsection{IDS for metric graphs}

Let $\left(\Omega,T\right)$ be a uniquely ergodic subshift, with
an associated family of metric decorated $\Z$-graphs $\Gamma_{\Omega}$.
For $\omega\in\Omega$ and $n\in\mathbb{N}$, we restrict $\Gamma_{\omega}$
to a compact graph by removing the edges $(-1,0)\cdot L$ and $(n,n+1)\cdot L$
from $\Gamma_{\omega}$, and denote by $\pa$ the resulting compact
connected component (see Figure \ref{fig: GammaN}). At the cut vertices,
$0$ and $nL$, impose Neumann-Kirchhoff vertex conditions (though
the results below do not depend on the vertex conditions as long as
they render the operator self-adjoint). The corresponding Kirchhoff
Laplacian $\Ha$ is bounded from below and has a compact resolvent,
and thus has purely discrete spectrum accumulating at infinity. Denote
the associated normalized spectral counting function by
\begin{equation}
N_{\omega}^{(n)}\left(E\right):=\frac{\#\left\{ \lambda\in\spec{\Ha}:\lambda\le E\right\} }{\left|\pa\right|}.\label{eq:truncation-IDS}
\end{equation}

\begin{figure}
\includegraphics[scale=0.55]{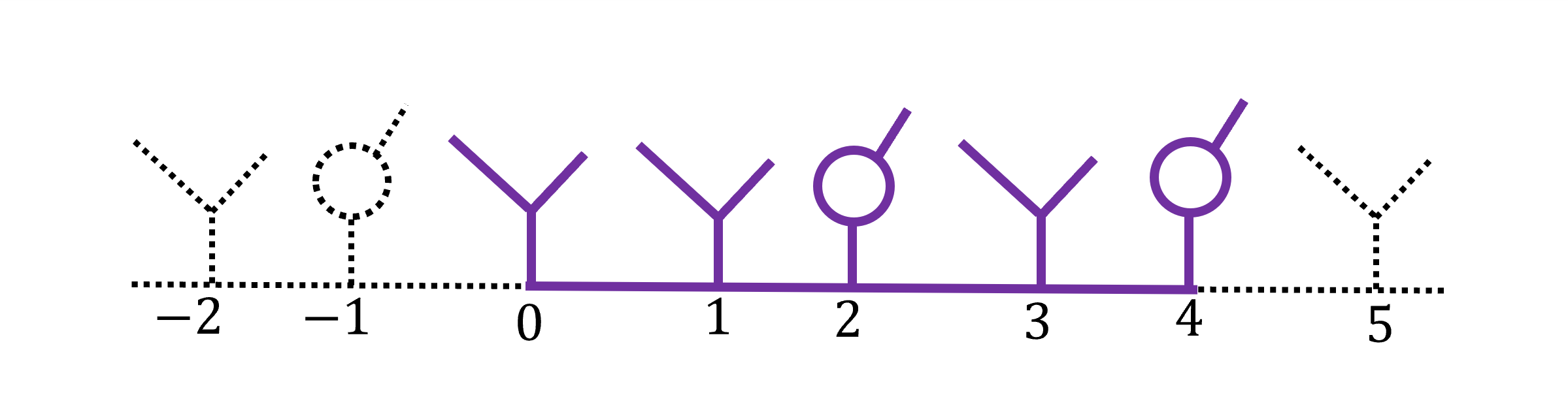}

\caption{The compact graph $\Gamma_{\omega}|_{\left[0,4\right]}$, constructed
by truncating the infinite graph $\Gamma_{\omega}$ and keeping five
decorations.\label{fig: GammaN}}
\end{figure}

\begin{prop}
\label{prop:IDS-existence}For almost all $\omega\in\Omega$, the
sequence of functions $N_{\omega}^{(n)}\left(E\right)$ converges
uniformly as $n\rightarrow\infty$ to a function $N_{\Omega}\left(E\right):\mathbb{R}\rightarrow\mathbb{R}$.
We call the function $N_{\Omega}\left(E\right)$ the integrated density
of states (IDS) of the family $H_{\Omega}$.
\end{prop}

The proof of the proposition above relies on an adaptation of a method
from \cite{GruLenVes_jfa07} and appears in Appendix \ref{sec:IDS-exists}.

The IDS is a nondecreasing function which is constant at each connected
component of the complement of the spectrum (called \textit{spectral
gaps}). We are interested in the gap labels of $N_{\Omega}$:
\begin{equation}
\mathcal{GL}\left(N_{\Omega}\right):=\left\{ N_{\Omega}\left(E\right):E\in\R\backslash\spec{H_{\Omega}}\right\} \subset\mathbb{R}.\label{eq:GL}
\end{equation}

\subsubsection{IDS for discrete graphs}

The IDS $N_{\Omega}^{\Delta}\left(E\right)$ for the discrete Laplacian
is defined similarly to the metric case discussed above. Remove from
$G_{\omega}$ the two edges $(-1,0)$ and $(n,n+1)$ and denote by
$\dpa$ the resulting compact connected component. The resulting operator
$\left.\Delta_{\omega}\right|_{[0,n]}$ is a self-adjoint matrix.
We define the IDS as the limit of associated normalized spectral counting
functions,
\begin{equation}
N_{\Omega}^{\Delta}\left(E\right):=\lim_{n\rightarrow\infty}\frac{\#\left\{ \lambda\in\spec{\left.\Delta_{\omega}\right|_{[0,n]}}:\lambda\le E\right\} }{\left|\V_{\dpa}\right|},\label{eq:truncation-IDS-2}
\end{equation}
where the limit exists for almost all $\omega\in\Omega$ and its value
is independent of $\omega$ (the proof is similar to that of Proposition\ \ref{prop:IDS-existence},
see Appendix\ \ref{sec:IDS-exists}).

\subsection{The Schwartzman group\label{subsec: Schwartzman-group}}

We now define the Schwartzman group, which plays a central role in
Theorems \ref{thm:GLT} and \ref{thm:Discrete-GLT}. For more details,
see \cite{DamEmiFil_jst23,damFil_book22,DamFil_otaa23,Damanik2023}
and references therein.

Let $\left(\Omega,T\right)$ be a uniquely ergodic subshift, equipped
with a (unique) invariant probability measure $\mu$. We associate
with this dynamical system a suspension space:
\begin{equation}
X_{\Omega}:=\Omega\times\left[0,1\right]/\left\{ \left(\omega,1\right)\sim\left(T\omega,0\right)\right\} .\label{eq:suspension}
\end{equation}
The space $X_{\Omega}$ is naturally endowed with the translation
flow in the second factor:
\begin{align}
 & \tau^{t}:X_{\Omega}\rightarrow X_{\Omega}\,\,\,\,\left(t\in\R\right),\label{eq:flow-1}\\
 & \tau^{t}\left(\omega,s\right)=\left(\omega,t+s\text{ mod \ensuremath{1}}\right),\label{eq:flow-2}
\end{align}
and with a probability measure $\eta$:
\begin{equation}
\int_{X_{\Omega}}fd\eta=\int_{0}^{1}\int_{\Omega}f\left(\left[\omega,t\right]\right)d\mu\left(\omega\right)dt.\label{eq:measure-prod}
\end{equation}

Let $C^{\sharp}\left(X_{\Omega}\right)$ be the space of homotopy
classes of functions from $X_{\Omega}$ to the one-dimensional torus
$\mathbb{T}:=\R/\Z$. For a given function $\phi:X_{\Omega}\rightarrow\T$,
let $\phi_{x}$ denote the restriction of $\phi$ to the orbit of
a point $x=\left(\omega,s\right)\in X_{\Omega}$ under the flow $\tau^{t}$:
\begin{align}
 & \phi_{x}:\R\rightarrow\T,\label{eq:lift1}\\
 & \phi_{x}\left(t\right)=\phi\left(\tau^{t}x\right).\label{eq:lift2}
\end{align}
Since $\R$ is the universal cover of $\T,$ the function $\phi_{x}$
is naturally lifted to a map $\widetilde{\phi}_{x}\left(t\right):\R\rightarrow\R$.
With this in mind, define the Schwartzman homomorphism by
\begin{align}
 & S_{\Omega}:C^{\sharp}\left(X_{\Omega}\right)\rightarrow\mathbb{R},\label{eq:SH1}\\
 & S_{\Omega}\left(\left[\phi\right]\right)=\lim_{t\rightarrow\infty}\frac{\widetilde{\phi}_{x}\left(t\right)}{t},\label{eq:SH2}
\end{align}
where the limit above is $\mu$ almost-surely independent of $\omega$
(where $x=(\omega,s)$) \cite[thm. 3.9.13]{damFil_book22}. In other
words, the Schwartzman homomorphism is the average rate of rotation
of $\phi_{x}$ along the flow.
\begin{defn}
The \emph{Schwartzman group} $\S$ is the image of $S_{\Omega}$.
\end{defn}

The Schwartzman groupis a countable subgroup of $\R$, which depends
on the full dynamical system $\left(\Omega,T,\mu\right)$, but for
brevity we denote it by $\S$.
\begin{example}
\label{exa:SG-Sturmian} Let $\alpha\in\left(0,1\right)\backslash\Q$
and let $\left(\Omega_{\alpha},T\right)$ be the Sturmian subshift
from Example \ref{exa:Sturmian}. Its Schwartzman group is given by
\begin{equation}
\mathfrak{S}_{\Omega_{\alpha}}=\set{\alpha n+m}{m,n\in\Z},\label{eq:SG-Sturm}
\end{equation}
see \cite[thm. 10.9.3]{DamFill_book_vol_2}.
\end{example}

\subsection{Main results\label{subsec:Main-results}}

Our first main result is a gap labelling theorem (GLT) for the metric
graph operator family $H_{\Omega}$.
\begin{thm}
\label{thm:GLT}Let $\left(\Omega,T\right)$ be a uniquely ergodic
subshift, with an associated family of metric decorated $\Z$-graphs
$\Gamma_{\Omega}$, equipped with the Kirchhoff Laplacian. Then,
\begin{equation}
\mathcal{GL}\left(N_{\Omega}\right)\subset\frac{1}{\overline{L}\left(\Gamma_{\Omega}\right)}\S\cap[0,\infty),\label{eq:GL1}
\end{equation}
where $\S$ is the Schwartzman group.
\end{thm}

The following is an immediate application of the theorem above to
the Sturmian subshift $\Omega_{\alpha}$:
\begin{cor}
\label{cor:GLT-Sturmian}For a Sturmian decorated $\Z$-graph, the
possible gap labels are given by
\begin{equation}
\mathcal{GL}\left(N_{\Omega_{\alpha}}\right)\subset\set{\frac{\alpha n+m}{L+\alpha\ell_{1}+\left(1-\alpha\right)\ell_{2}}}{m,n\in\Z}\cap[0,\infty),\label{eq:AP-GL}
\end{equation}
where $\ell_{1},\ell_{2}$ are the total lengths of the decorations,
and $L$ is the horizontal distance between the decorations.
\end{cor}

Our next main result is a GLT for discrete decorated $\Z$-graphs:
\begin{thm}
\label{thm:Discrete-GLT} Let $\left(\Omega,T\right)$ be a uniquely
ergodic subshift, with an associated family of discrete decorated
$\Z$-graphs $G_{\Omega}$, equipped with the normalized discrete
Laplacian. Then
\begin{equation}
\mathcal{GL}\left(N_{\Omega}^{\Delta}\right)\subset\frac{1}{\overline{V}\left(G_{\Omega}\right)}\S\cap\left[0,1\right].\label{eq:DGL}
\end{equation}
\end{thm}

As we shall see, for some non-generic choices of the edge lengths,
$\spec{H_{\Omega}}$ may contain isolated eigenvalues, corresponding
to jump discontinuities in the IDS. For Sturmian comb graphs (defined
in Example~\ref{exa:Sturmian comb}), these eigenvalues and IDS jumps
are fully characterized:
\begin{thm}
\label{thm:IDS-jumps}Let $\Gamma_{\Omega_{\alpha}}$ be a metric
Sturmian comb graph.

If $E\in\spec{H_{\Omega_{\alpha}}}$ is an eigenvalue then $E$ is
of infinite multiplicity, with compactly supported eigenfunctions.

The resulting jump in the IDS takes one of the following values
\begin{equation}
\Delta N_{\Omega_{\alpha}}\left(E\right)\in\frac{1}{\ell\alpha+L}\left\{ (c_{1}+1)\alpha-1,~-c_{1}\alpha+1,~\alpha\right\} ,\label{eq:IDS-JUMPS}
\end{equation}
where $\ell$ is the length of the decoration, $L$ is the decoration
spacing, and $c_{1}$ is the first digit in the continued fraction
expansion of $\alpha$:
\begin{equation}
\alpha=\frac{1}{c_{1}+\frac{1}{c_{2}+...}}.\label{eq:cf-expansion}
\end{equation}
\end{thm}

A more detailed version of Theorem \ref{thm:IDS-jumps}, which also
contains the explicit expressions for the eigenvalues, is presented
as Theorem~\ref{thm: IDS jumps Sturmian combs}.

\subsection*{Acknowledgments}

We would like to thank Siegfried Beckus, David Damanik, Jake Fillman,
Johannes Kellendonk, Daniel Lenz, and Jan Mazá\v{c} for extremely
helpful discussions and feedback. We gratefully acknowledge the hospitality
of the Institute of Mathematics at the University of Potsdam, where
some of this research took place. The research for this paper was
partially conducted at the Israel Institute for Advanced Studies,
as part of the Research Group Analysis, Geometry, and Spectral Theory
of Graphs during 2025-2026. This research was supported by the Israel
Science Foundation (ISF Grant No. 2362/25), and by the United States
- Israel Binational Science Foundation (BSF), grant no. 2024154.

\section{Gap labelling for metric graphs\label{sec:Gap-labeling} - Proof
of theorem~\ref{thm:GLT}}

We begin by introducing the nodal count of metric graphs, which is
later used for the proof of Theorem~\ref{thm:GLT}.

\subsection{The nodal surplus of a tile\label{subsec:The-nodal-surplus}}

Let $\gra$ be a decoration of type $a\in\A$ with base vertex $v_{a}$.
For an energy $E\in\mathbb{R}$, consider the differential equation
on $\gra$,
\begin{equation}
-\frac{d^{2}}{dx^{2}}f=Ef,\label{eq:E-ODE}
\end{equation}
subject to the Kirchhoff conditions at all vertices of $\gra$, except
for $\va$, where we impose only a continuity condition (\ref{eq:-15-2}),
but no condition on the derivatives. Then for all but a discrete subset
of $E\in\R$, this equation has a unique solution (up to a scalar
multiple), denoted by $\fe$. This discrete set is exactly the spectrum
of the Kirchhoff Laplacian on $\Gamma_{a}$ with a Dirichlet condition
imposed at $v_{a}$ (see e.g., \cite[thm. 2.1 and cor. 2.4]{BanBerSmi_ahp12}).
For $E$ values outside of this discrete set we denote
\begin{equation}
m_{a}\left(E\right):=\sum_{e\in\E_{v_{a}}}\frac{\fe'|_{e}\left(\va\right)}{\fe\left(\va\right)}.\label{eq:m-func}
\end{equation}
We use this to define an (energy dependent) Robin vertex condition
at $\va$:
\begin{align}
 & g|_{e}\left(v_{a}\right)=g|_{e'}\left(v_{a}\right),\quad\quad\forall e,e'\in\E_{v_{a}},\label{eq:m-Robin-1}\\
 & \sum_{e\in\E_{v_{a}}}g'|_{e}\left(\va\right)=m_{a}\left(E\right)g\left(\va\right).\label{eq:Robin2}
\end{align}
By construction, $\left(E,\fe\right)$ is an eigenpair of $-\frac{d^{2}}{dx^{2}}$
on $\gra$, with Kirchhoff condition imposed at all vertices except
for $\va$, where the Robin condition (\ref{eq:m-Robin-1}), (\ref{eq:Robin2})
is imposed. We denote the resulting operator by $\left.H\right|_{\Gamma_{a}}$,
keeping in mind that this operator depends on $E$ (but do not indicate
this in the notation for brevity).

Denoting the (non-normalized) spectral counting function of $\left.H\right|_{\Gamma_{a}}$
by
\begin{equation}
n^{\left(a\right)}\left(E\right):=\#\left\{ \lambda\in\spec{\left.H\right|_{\Gamma_{a}}}:\lambda\le E\right\} ,\label{eq:counting-1}
\end{equation}
we define the nodal surplus of $E$ in $\gra$ by
\begin{equation}
\sigma^{\left(a\right)}\left(E\right):=\#\left\{ \text{zeros of \ensuremath{f_{E}} in \ensuremath{\gra}}\right\} -\left(n^{\left(a\right)}\left(E\right)-1\right).\label{eq:surplus}
\end{equation}
Outside a discrete set of $E$ values, the eigenfunction $f_{E}$
does not vanish at any vertex of $\gra$ \cite[cor. 2.4]{BanBerSmi_ahp12}
(This discrete set contains the discrete set mentioned above, and
in general might be larger). Restricting $E$ to be outside the mentioned
set, and using the unique continuation of $f_{E}$ at every edge of
$\gra$, we conclude that the zero set of $f_{E}$ is discrete. Hence
the surplus $\sigma^{\left(a\right)}\left(E\right)$ is well-defined
for all such $E$ values.

The nodal surplus has been extensively studied for quantum graphs
starting from \cite{GnuSmiWeb_wrm04}. For additional background see
\cite{AloBanBer_cmp18,AloBanBer_exp22,Alon_PhDThesis,BanBerSmi_ahp12,Ber_cmp08}.

\subsection{Right propagation along $\Gamma_{\omega}$}

For the proof in the next subsection we need to establish a notion
of propagation along $\Gamma_{\omega}$. Fix $\omega\in\Omega$, and
set the origin $o\left(\Gamma_{\omega}\right)$ to be the vertex with
$0$ coordinate of the $\Z$-graph (which is identified with the base
vertex of $\Gamma_{\omega(0)}$).

\begin{figure}
\includegraphics[scale=0.55]{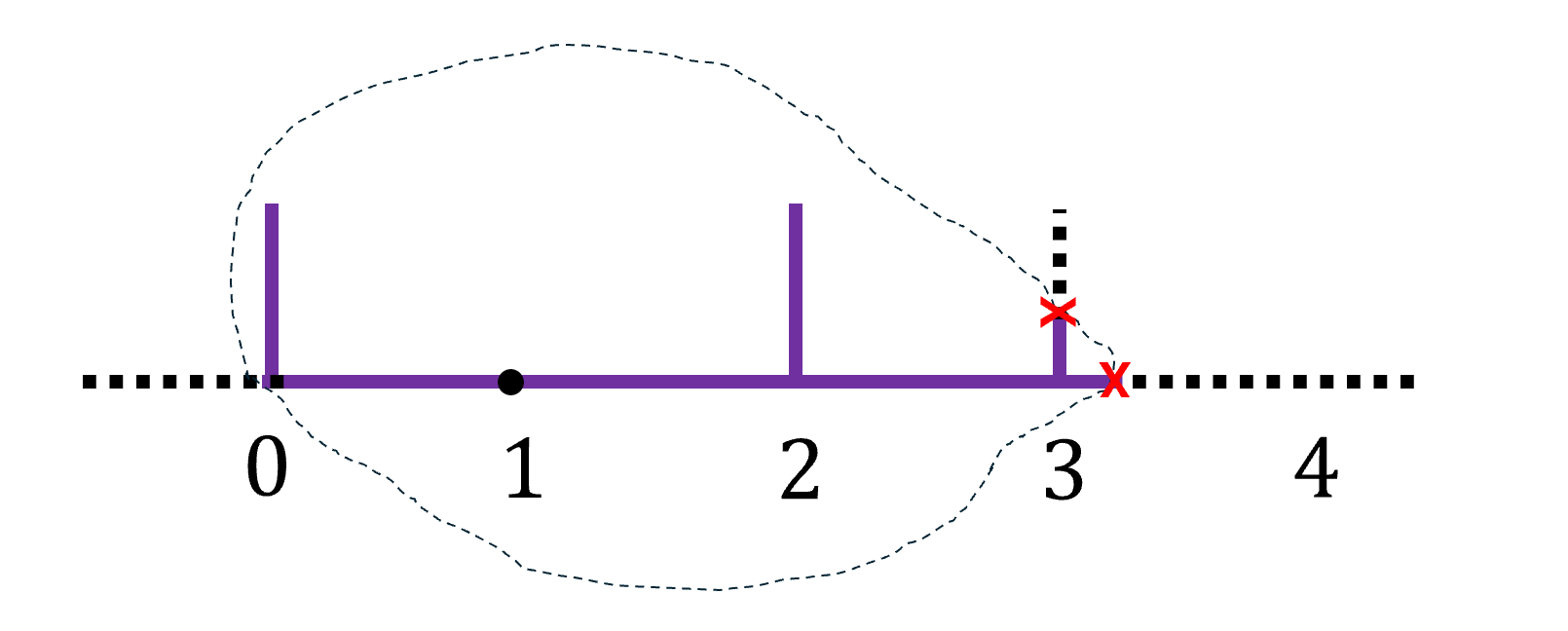}

\caption{The compact graph $\Gamma_{\omega}\left(t\right)$ (\ref{eq: truncated-graph}).
Here, $s_{\omega}\left(t\right)$ (\ref{eq: spheres def}) consists
of two points, marked by x signs. \label{fig: HorizontalCovering}}
\end{figure}

Assume first that for all $a\in\A$ the total metric length of $\gra$
is smaller than $L$ (the horizontal distance between adjacent decorations).
Under this assumption we set for $t\geq0$, 
\begin{equation}
s_{\omega}\left(t\right):=\set{x\in\Gamma_{\omega}^{+}}{d\left(o\left(\Gamma_{\omega}\right),x\right)=tL},\label{eq: spheres def}
\end{equation}
see Figure \ref{fig: HorizontalCovering}, where $\Gamma_{\omega}^{+}$
is the right part of $\Gamma_{\omega}$, i.e., the half positive ray,
$\left[0,\infty\right)$, together with all decorations attached to
it. In particular, we note that $s_{\omega}\left(0\right)=\left\{ o\left(\Gamma_{\omega}\right)\right\} $
and for $k\in\N$, $s_{\omega}\left(k\right)=\left\{ kL\right\} $.
We wish to maintain the property $s_{\omega}\left(k\right)=\left\{ kL\right\} $
at integer radii even when some decorations $\gra$ have total length
exceeding $L$. To do so, if needed, we can rescale the metric inside
each decoration $\Gamma_{a}$ used in (\ref{eq: spheres def}) by
a factor of $\frac{L}{2\left|\gra\right|}$, while leaving the metric
unchanged at the horizontal $\Z$-graph.

\subsection{Proof of Theorem \ref{thm:GLT}}

The proof proceeds in three main steps: first, for each gap of $H_{\omega}$
we define an appropriate function on the suspension space $X_{\Omega}$
for which the Schwartzman homomorphism will be evaluated. Second,
we relate this function to the nodal count of a generalized eigenfunction
and finally, we combine these results to express the IDS value at
the gap in terms of the Schwartzman group.

\subsubsection*{\textbf{Step one: Defining an appropriate function on the suspension.}}

Fix $E\in\R\backslash\spec{H_{\Omega}}$. For $\omega\in\Omega$,
consider the differential equation on $\Gamma_{\omega}^{+}$, 
\begin{equation}
-\frac{d^{2}}{dx^{2}}u\left(x\right)=Eu\left(x\right),\label{eq:E-ODE-1}
\end{equation}
with the Kirchhoff condition imposed at all vertices, except the origin
where no boundary condition is imposed. Since $E\notin\spec{H_{\omega}}$,
this equation has a unique solution (up to a scalar multiple) which
is in $L^{2}(\Gamma_{\omega}^{+})$, denoted $\fwe$ (the proof is
similar to that of \cite[lem. 9.7]{Teschl2014}). In addition, the
uniqueness of the solution guarantees that, up to normalization, $\left.f_{T\omega,E}\right|_{T\Gamma_{\omega}^{+}}=\left.T\fwe\right|_{T\Gamma_{\omega}^{+}}$,
where $T$ acts on $\fwe$ as in (\ref{eq:Koopman}). Each solution
$f_{\omega,E}$ may be also extended to the left (i.e., to $\Gamma_{\omega}\backslash\Gamma_{\omega}^{+}$)
by solving the ODE (\ref{eq:E-ODE-1}), starting from the initial
conditions given by the value and derivative of $\fwe$. We hence
may adopt the notation $f_{\omega,E}$ for a function on the whole
$\Gamma_{\omega}$ and we get that $\left.f_{T\omega,E}\right|_{\Gamma_{\omega}^{+}}\in L^{2}(\Gamma_{\omega}^{+})$
and $f_{T\omega,E}=T\fwe$.

We use the function $f_{\omega,E}$ to define a function from the
suspension space to the one-dimensional torus. Consider the following
form of the Cayley transform, 
\begin{equation}
\mathop{C}(t)=\frac{t+\rmi}{t-\rmi},\label{eq:Cayley}
\end{equation}
which maps the left to right oriented real line $\overline{\R}$ (augmented
with $\pm\infty$) onto the clockwise oriented unit circle. Using
this we define the following function on the suspension space:

\begin{align}
 & \phi:X_{\Omega}\rightarrow\T,\label{eq:phi1}\\
 & \phi\left(\omega,t\right)=\frac{1}{2\pi}\mathop{Arg}\left[\mathop{C}\left(\sum_{x\in s_{\omega}\left(t\right)}\frac{\fwe'\left(x\right)}{\fwe\left(x\right)}\right)\right],\label{eq:phi2}
\end{align}
where $s_{\omega}\left(t\right)$ is given in (\ref{eq: spheres def}),
and $\mathop{Arg}$ is the argument function mapping complex numbers
onto the one dimensional torus $\left[0,2\pi\right)$. In the sum
over $x\in s_{\omega}\left(t\right)$ above, a special emphasis should
be given to the case when $x$ is a vertex. The derivative $\fwe'\left(x\right)$
at a vertex $x$ is defined by parameterizing the elements of $s_{\omega}\left(t\right)$
as $x(t)$ and setting $\fwe'\left(x(t)\right):=\lim_{\tilde{t}\rightarrow t^{-}}\fwe'\left(x\left(\tilde{t}\right)\right)$.
For $t=0$ and $x=o(\Gamma)$, we similarly set $\fwe'\left(x\right):=\lim_{\tilde{x}\rightarrow x^{-}}\fwe'\left(\tilde{x}\right)$.
The function $\phi$ may be considered as a generalized Pr\"ufer
angle. Straightforward computation shows that $\phi$ is well-defined
on $X_{\Omega}$, as
\begin{align}
\phi\left(\omega,1\right) & =\frac{1}{2\pi}\mathop{Arg}\left[\mathop{C}\left(\sum_{x\in s_{\omega}\left(1\right)}\frac{\fwe'\left(x\right)}{\fwe\left(x\right)}\right)\right]\label{eq:phi-invariance}\\
 & =\frac{1}{2\pi}\mathop{Arg}\left[\mathop{C}\left(\sum_{x\in s_{T\omega}\left(0\right)}\frac{f_{T\omega,E}'\left(x\right)}{f_{T\omega,E}\left(x\right)}\right)\right]=\phi\left(T\omega,0\right),\nonumber 
\end{align}
where we have used that $f_{T\omega,E}=T\fwe$ and also the equivalence
between $s_{\omega}\left(1\right)$ in $\Gamma_{\omega}$ and $s_{T\omega}\left(0\right)$
in $\Gamma_{T\omega}$ (both consist of a single point, which is the
same up to the isomorphism between $\Gamma_{\omega}$ and $\Gamma_{T\omega}$).
We further argue that $\phi$ is continuous on $X_{\Omega}$. First,
we show that $\phi(\omega,t)$ is continuous in $\omega\in\Omega$
using the Titchmarsh-Weyl $m$-function of the half infinite graph
$\Gamma_{\omega}^{+}$, denoted $m_{\omega}^{+}(z)$. . The values
$\fwe\left(x\right)$ and $\fwe'\left(x\right)$ depend continuously
on the Robin boundary condition $\frac{\fwe'\left(o(\Gamma_{w}^{+})\right)}{\fwe\left(o(\Gamma_{w}^{+})\right)}$
at the origin, as solutions of the ODE (\ref{eq:E-ODE-1}). Therefore
the RHS of (\ref{eq:phi2}) depends continuously on $m_{\omega}^{+}(E)=\frac{\fwe'\left(o(\Gamma_{w}^{+})\right)}{\fwe\left(o(\Gamma_{w}^{+})\right)}$,
and since the $m$-function $m_{\omega}^{+}(E)$ is continuous in
$\omega$ (see e.g., \cite{BanSof_prep24a}), we conclude that $\phi(\omega,t)$
is continuous in $\omega$. Next, we show that $\phi(\omega,t)$ is
also continuous in $t$. Clearly the expression $\sum_{x\in s_{\omega}\left(t\right)}\frac{\fwe'\left(x\right)}{\fwe\left(x\right)}$
is continuous in $t$, when $s_{\omega}\left(t\right)$ does not contain
any vertex of $\Gamma_{\omega}^{+}$. In addition, the Kirchhoff vertex
conditions (\ref{eq:-15-2}),(\ref{eq:-16-2}) ensure the continuity
of $\sum_{x\in s_{\omega}\left(t\right)}\frac{\fwe'\left(x\right)}{\fwe\left(x\right)}$
in $t$ also when $s_{\omega}\left(t\right)$ contains a vertex. Overall,
we conclude that the function $\phi$ is well-defined and continuous
on $X_{\Omega}$.

\subsubsection*{\textbf{Step 2: Expressing $\phi\left(\omega,t\right)$ using the
nodal count of $\protect\fwe$.}}

Having defined $\phi:X_{\Omega}\rightarrow\T$ in (\ref{eq:phi2})
we wish to apply the Schwartzman homomorphism to it via (\ref{eq:SH2}).
At this point, fix $\omega\in\Omega$ to be in the full measure set
for which (\ref{eq:SH2}) holds. Define
\begin{equation}
\Gamma_{\omega}\left(t\right):=\set{x\in\Gamma_{\omega}^{+}}{d\left(o\left(\Gamma_{\omega}\right),x\right)\leq tL},\label{eq: truncated-graph}
\end{equation}
see Figure \ref{fig: HorizontalCovering}, and note that $s_{\omega}\left(t\right)$
forms part of the boundary of $\Gamma_{\omega}\left(t\right)$. By
(\ref{eq:phi2}) the function $\phi\left(\omega,t\right)$ equals
$0\in\T$ precisely when $\fwe\left(x\right)=0$ for some $x\in s_{\omega}\left(t\right)$.
With this observation we use the values of $\phi\left(\omega,t\right)$
(or more precisely its lift) to count the zeros of $\fwe$. To do
so, recall the notation $\phi_{(\omega,0)}(t):=\phi(\tau^{t}(\omega,0))$
and $\widetilde{\phi}_{(\omega,0)}(t)$ for its lift (see Section~\ref{subsec: Schwartzman-group}).
With this notation, the number of zeros of $\fwe$ in $\Gamma_{\omega}\left(t\right)$
is equal to the number of times that the function $\widetilde{\phi}_{(\omega,0)}$
intersects $0\text{ mod }1$ in the interval $\left[0,t\right]$.
We use this observation to connect between the (average) zero count
of $\fwe$ and the value of the Schwarzman homomorphism $S_{\Omega}\left(\left[\phi\right]\right)$.
Explicitly, using the notation
\begin{equation}
\zc:=\#\set{x\in\Gamma_{\omega}\left(t\right)}{\fwe(x)=0},\label{eq:Zwt}
\end{equation}
 we have 
\begin{equation}
\zc=\left\lfloor \widetilde{\phi}_{(\omega,0)}(t)\right\rfloor ,\label{eq:Zwt-1}
\end{equation}
where $\left\lfloor \phantom{x}\right\rfloor $ denotes the floor
function. Therefore, by (\ref{eq:SH2}),

\begin{equation}
S_{\Omega}\left(\left[\phi\right]\right)=\lim_{t\rightarrow\infty}\frac{\tilde{\phi}_{\left(\omega,0\right)}\left(t\right)}{t}=\lim_{t\rightarrow\infty}\frac{1}{t}\left\lfloor \widetilde{\phi}_{(\omega,0)}(t)\right\rfloor =\lim_{t\rightarrow\infty}\frac{1}{t}\zc,\label{eq: Schwarz equals zero count}
\end{equation}
where in the first equality we used that $\omega\in\Omega$ is in
the full measure set for which (\ref{eq:SH2}) holds.

Having this connection between the Schwartzman homomorphism and the
nodal count, we analyze $Z_{\omega,t}$. In what follows we decompose
the total nodal count on $\Gamma_{\omega}(t)$ via the nodal count
of its subgraphs: the decorations, and the horizontal path. Outside
a discrete set of energies $E$, the solution to the ODE (\ref{eq:E-ODE})
on each decoration $\Gamma_{a}$ is unique up to scalar multiple (as
discussed in Subsection~\ref{subsec:The-nodal-surplus}). Hence,
the nodal count on each decoration of a given type does not depend
on the location of this decoration within $\Gamma_{\omega}(t)$. Denoting
this nodal count function by $Z^{\left(a\right)}(E)$, we write
\begin{equation}
\zc=\zch+\sum_{a\in\mathcal{A}}\ct t\zca,\label{eq:nodal-split}
\end{equation}
where $\zch$ is the nodal count function of $\fwe$ on the path graph
$\left[0,tL\right]$ (which is a subgraph of $\Gamma_{\omega}(t)$),
and we extend the definition of the letter counting function (\ref{eq:counting})
to non-integer $t$ values by setting $\ct t:=\ct{\left\lfloor t\right\rfloor }$
to be the number of decorations of type $a$ in $\Gamma_{\omega}(t)$.
Note that (\ref{eq:nodal-split}) is an equality between functions
in $E$, but for brevity we omit the $E$-dependence. As already mentioned,
these functions are well-defined up to a discrete set of $E$ values.
We further use the spectral counting functions (\ref{eq:counting-1})
and nodal surplus functions (\ref{eq:surplus}) to write
\begin{align}
\zc= & \zch+\sum_{a\in\mathcal{A}}\ct t\left(n^{\left(a\right)}+\sigma^{\left(a\right)}-1\right),\label{eq:total-zeros}
\end{align}

We next express the nodal counting functions $\zc$ and $\zch$ through
spectral counting functions of suitable operators. Towards this, we
define the corresponding operators. First, consider the restriction
of $H_{\omega}$ to the finite graph $\Gamma_{\omega}(t)$. At the
vertices $u\in s_{\omega}(t)\cup o(\Gamma_{\omega})$ we impose the
Robin condition
\begin{equation}
\frac{f'\left(u\right)}{f\left(u\right)}=\frac{\fwe'\left(u\right)}{\fwe\left(u\right)},\label{eq:Robin-boundaries}
\end{equation}
and at all other vertices of $\Gamma_{\omega}(t)$ we impose the Neumann-Kirchhoff
vertex conditions as in $H_{\omega}$. We naturally denote the resulting
operator $\left.H_{\omega}\right|_{\Gamma_{\omega}(t)}$. We describe
now an operator associated with the horizontal subgraph $\left[0,tL\right]$.
Let $v_{m}$ be an interior vertex of $\left[0,tL\right]$, which
is positioned at $mL$, where $m\in\Z\cap(0,t)$. Denote its two neighboring
edges by $e_{m}^{\pm}$. We impose at the vertex $v_{m}$ the Robin-type
conditions
\begin{align}
 & f|_{e_{v}^{+}}\left(v_{m}\right)=f|_{e_{v}^{-}}\left(v_{m}\right)=:f\left(v_{m}\right),\label{eq:Robin-internal-1}\\
 & f'|_{e_{v}^{+}}\left(v_{m}\right)+f'|_{e_{v}^{-}}\left(v_{m}\right)=-m_{\omega(m)}\left(E\right)f\left(v_{m}\right),\label{eq:Robin-internal-2}
\end{align}
where the Robin parameter $m_{\omega(m)}\left(E\right)$ is as in
(\ref{eq:m-func}), and takes into account that in $\Gamma_{\omega}(t)$
the decoration $\Gamma_{\omega(m)}$ is glued to $v_{m}$. At the
boundary vertices $u\in\{o(\Gamma_{\omega}),tL\}$ we impose the same
Robin conditions (\ref{eq:Robin-boundaries}) as were imposed for
$\left.H_{\omega}\right|_{\Gamma_{\omega}(t)}$. Overall these vertex
conditions render the one-dimensional Laplacian on $[0,tL]$ a self-adjoint
operator, which we denote by $\left.H_{\omega}\right|_{[0,tL]}$.
These particular choices of vertex conditions guarantee that $\left(E,\left.\fwe\right|_{[0,tL]}\right)$
is an eigenpair of $\left.H_{\omega}\right|_{[0,tL]}$ and $\left(E,\left.\fwe\right|_{\Gamma_{\omega}(t)}\right)$
is an eigenpair of $\left.H_{\omega}\right|_{\Gamma_{\omega}(t)}$.

Denoting the spectral counting function of $\left.H_{\omega}\right|_{[0,tL]}$
by 
\begin{equation}
\sch:=\#\set{\lambda\in\mathrm{Spec}\left(\left.H_{\omega}\right|_{[0,tL]}\right)}{\lambda\leq E},\label{eq:n-horiz}
\end{equation}
Sturm's oscillation theorem (see \cite{Ber_cmp08} and \cite{Sch_wrcm06})
yields
\begin{equation}
\sch(E)=\zch(E)+1.\label{eq:counting=00003Dzeros}
\end{equation}
 Substituting this in (\ref{eq:total-zeros}) gives 
\begin{equation}
\zc=\sch-1+\sum_{a\in\mathcal{A}}\ct t\left(n^{\left(a\right)}+\sigma^{\left(a\right)}-1\right).\label{eq: nodal_count_via_spectral_count}
\end{equation}

We next relate the spectral counting functions of the three operators
$\left.H_{\omega}\right|_{\Gamma_{\omega}(t)}$, $\left.H_{\omega}\right|_{[0,tL]}$,
$\left.H_{\omega}\right|_{\Gamma_{a}}$ discussed above (the operator
$\left.H_{\omega}\right|_{\Gamma_{a}}$ was presented in Section \ref{subsec:The-nodal-surplus},
where it was denoted by $\left.H\right|_{\Gamma_{a}}$).
\begin{lem}
\label{lem:Counting-lemma}Let $E\notin\spec{H_{\omega}}$. Assume
that for all $a\in\A$, the spectrum of the Kirchhoff Laplacian on
$\Gamma_{a}$ with Dirichlet condition imposed at $v_{a}$ does not
contain $E$. Denote the spectral counting functions of $\left.H_{\omega}\right|_{\Gamma_{\omega}(t)}$,
$\left.H_{\omega}\right|_{[0,tL]}$ and $\left.H_{\omega}\right|_{\Gamma_{a}}$
by $\sc$ , $\sch$ and $\sca$ respectively. Then,
\begin{equation}
\sc\left(E\right)=\sch\left(E\right)+\sum_{a\in\mathcal{A}}\ct t\left(\sca\left(E\right)-1\right).\label{eq:counting-lemma}
\end{equation}
\end{lem}

The proof of the lemma involves a continuous interpolation between
the relevant operators. While this is an interesting method, the proof
is somewhat technical and is postponed to Appendix~\ref{sec:Proof-of-counting-lemma-1}.

\begin{proof}[\textbf{\textit{Step 3: Computing the Schwartzman homomorphism of
$\phi$.}}]
 Using Lemma~\ref{lem:Counting-lemma}, Equation~(\ref{eq: nodal_count_via_spectral_count})
gives
\begin{equation}
\zc=\sc-1+\sum_{a\in\mathcal{A}}\ct t\sigma^{\left(a\right)}.\label{eq:counting-lemma-2}
\end{equation}
Now, computing the Schwartzman homomorphism as in (\ref{eq: Schwarz equals zero count})
gives 
\begin{align}
S_{\Omega}\left(\left[\phi\right]\right) & =\lim_{t\rightarrow\infty}\frac{1}{t}\zc(E)\nonumber \\
 & =\lim_{t\rightarrow\infty}\frac{1}{t}\left(\sc\left(E\right)-1+\sum_{a\in\mathcal{A}}\ct t\sigma^{\left(a\right)}\left(E\right)\right)\nonumber \\
 & =\lim_{t\rightarrow\infty}\frac{\sc\left(E\right)-1}{t}+\sum_{a\in\mathcal{A}}\lim_{t\rightarrow\infty}\frac{\ct t}{t}\sigma^{\left(a\right)}\left(E\right)\nonumber \\
 & =\lim_{t\rightarrow\infty}\left[\left(\frac{\sc\left(E\right)}{\left|\Gamma_{\omega,t}\right|}-\frac{1}{\left|\Gamma_{\omega,t}\right|}\right)\cdot\frac{\left|\Gamma_{\omega,t}\right|}{t}\right]+\sum_{a\in\mathcal{A}}\freq a\sigma^{\left(a\right)}\left(E\right)\nonumber \\
 & =N_{\Omega}\left(E\right)\cdot\overline{L}\left(\Gamma_{\Omega}\right)+\sum_{a\in\mathcal{A}}\freq a\sigma^{\left(a\right)}\left(E\right).\label{eq:SH-compute}
\end{align}
We thus finally obtain
\begin{equation}
N_{\Omega}\left(E\right)=\frac{S_{\Omega}\left(\left[\phi\right]\right)-\sum_{a\in\mathcal{A}}\freq a\sigma^{\left(a\right)}\left(E\right)}{\overline{L}\left(\Gamma_{\Omega}\right)}.\label{eq:GL-SH}
\end{equation}

To complete the proof we need to show that the numerator of (\ref{eq:GL-SH})
belongs to the Schwartzman group, $\S$. By definition this group
is the image of the Schwartzman homomorphism so that $S_{\Omega}\left(\left[\phi\right]\right)\in\S$.
Since $\S$ is an additive group, it is left to prove that $\sum_{a\in\mathcal{A}}\freq a\sigma^{\left(a\right)}\left(E\right)\in\S$.
From \cite[thm. 7.1]{DamFil_otaa23}, we know that $\S$ is the $\Z$-module
generated by
\begin{equation}
\left\{ \mu\left(\Xi\right):\Xi\text{ is a cylinder set in }\Omega\right\} ,\label{eq:cylinders}
\end{equation}
where a cylinder set is a subset of $\Omega$, for which a finite
subword is fixed to be a given value. In particular, we consider cylinder
sets with a single letter being fixed, which are of the form
\begin{equation}
\Xi_{a}:=\left\{ \omega\in\Omega:\omega\left(0\right)=a\right\} ,a\in\A.\label{eq:letter-cylinder}
\end{equation}
Since $\mu\left(\Xi_{a}\right)=\nu_{a}$ for a uniquely ergodic subshift,
and $\sigma^{\left(a\right)}\left(E\right)$ is an integer for all
$a\in\A$, we get $\sum_{a\in\mathcal{A}}\freq a\sigma^{\left(a\right)}\left(E\right)\in\S$,
as required.
\end{proof}

\section{Gap labelling for discrete graphs - Proof of Theorem~\ref{thm:Discrete-GLT}\label{sec:Discrete-GLT}}

In this section we prove the gap labelling theorem for discrete decorated
graphs (Theorem~\ref{thm:Discrete-GLT}). The main tool is the well-known
spectral relation between the discrete Laplacian and the Kirchhoff
Laplacian on the corresponding equilateral metric graph, summarized
below.
\begin{thm}
\label{thm:discrete metric} Let $\Gamma$ be an equilateral metric
graph with all edge lengths equal to $1$, equipped with the Kirchhoff
Laplacian $H$. Let $G$ be the associated discrete graph, equipped
with the normalized discrete Laplacian $\Delta$.
\begin{enumerate}
\item \label{enu: thm-discrete metric-1} For all $k\notin\left\{ \pi m:m\in\N\right\} $,
\begin{equation}
k^{2}\in\spec H\iff1-\cos\left(k\right)\in\spec{\Delta}.\label{eq:discrete-metric}
\end{equation}
Furthermore, if the corresponding points in the spectrum ($k^{2}$
and $1-\cos(k)$) are eigenvalues, then they have the same multiplicities.
\item \label{enu: thm-discrete metric-2} If, in addition, $\Gamma$ is
compact and connected, then its spectral counting function at $k^{2}=\pi^{2}m^{2}$
equals
\begin{equation}
\#\set{\lambda\in\spec H}{\lambda\leq\pi^{2}m^{2}}=\left|\E_{\Gamma}\right|m+M,\label{eq:spectral-counting}
\end{equation}
where $M\in\{0,1\}$ is the multiplicity of $1-\cos\left(\pi m\right)\in\{0,2\}$
in $\spec{\Delta}$.
\end{enumerate}
\end{thm}

The first part of the theorem is standard (see e.g., \cite{Cat_mm97,LlePos_jmaa08,Pan_lmp06,Bel_laa85}).
The second part follows from the case-by-case eigenvalue count in
\cite[prop. 6.2]{LlePos_jmaa08}, together with some basic properties
of the normalized discrete Laplacian.

Using Theorem \ref{thm:discrete metric}, we relate the IDS of the
discrete and metric decorated graphs. The ``conversion factor''
which connects between the discrete and metric IDS is given by
\begin{equation}
C\left(G_{\Omega}\right):=\frac{\overline{\V}\left(G_{\Omega}\right)}{\overline{\E}\left(G_{\Omega}\right)}=\frac{\sum_{a\in\A}\freq a\left|\V_{G_{a}}\right|}{1+\sum_{a\in\A}\freq a\left|\E_{G_{a}}\right|},\label{eq:conversion-ratio}
\end{equation}
which is the ratio between the average number of vertices and the
average number of edges.
\begin{prop}
\label{prop:counting-functions}Let $(\Omega,T)$ be a uniquely ergodic
subshift. Let $\left\{ \Gamma_{\omega}\right\} _{\omega\in\Omega}$
be a family of decorated $\Z$-graphs, such that each $\Gamma_{\omega}$
is an equilateral graph with all edge lengths equal to $1$. Let $\left\{ G_{\omega}\right\} _{\omega\in\Omega}$
be the associated discrete graphs. Denote the corresponding IDS functions
by $N_{\Omega}^{H}\left(E\right),N_{\Omega}^{\Delta}\left(E\right)$.
Then at every point $E$, where $N_{\Omega}^{H}$ is continuous we
have
\begin{equation}
N_{\Omega}^{H}\left(E\right)=\left\lfloor \frac{\sqrt{E}}{\pi}\right\rfloor +C\left(G_{\Omega}\right)\cdot\begin{cases}
N_{\Omega}^{\Delta}\left(1-\cos\left(\sqrt{E}\right)\right), & \left\lfloor \frac{\sqrt{E}}{\pi}\right\rfloor \text{ is even,}\\
\\1-N_{\Omega}^{\Delta}\left(1-\cos\left(\sqrt{E}\right)\right), & \left\lfloor \frac{\sqrt{E}}{\pi}\right\rfloor \text{ is odd.}
\end{cases}\label{eq:disc-metric-relation}
\end{equation}
\end{prop}

\begin{proof}
We first relate the spectral counting functions of compact discrete
and metric graphs. We then take the limit as in Proposition~\ref{prop:IDS-existence}
and (\ref{eq:truncation-IDS-2}) in order to compare the corresponding
IDS.

Let $\Gamma$ be an equilateral compact metric graph with all edge
lengths equal to $1$, and equipped with the Kirchhoff Laplacian $H$.
Let $G$ be the associated discrete graph equipped with $\Delta$.

\begin{figure}
\includegraphics[scale=0.45]{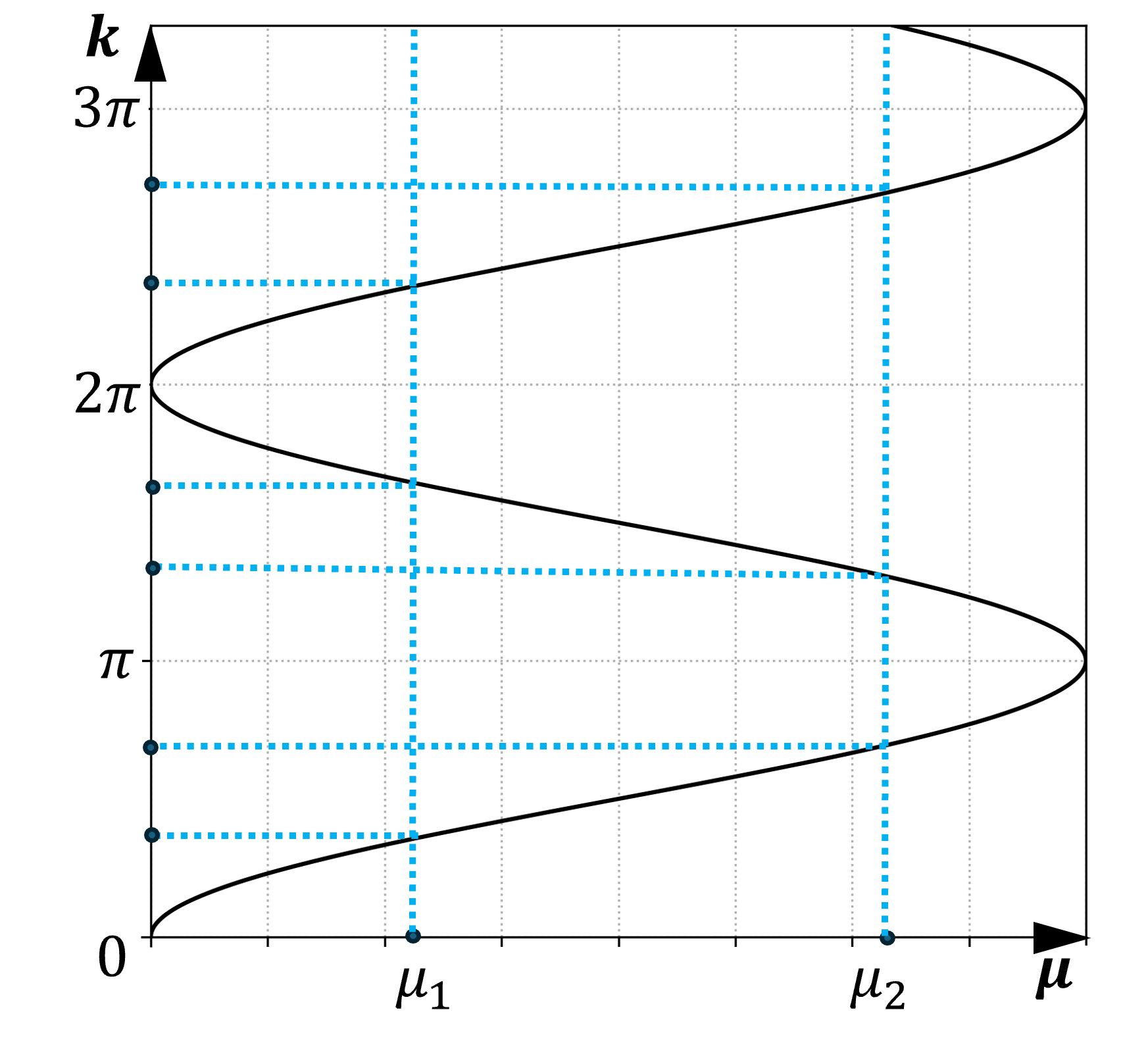}

\caption{The dispersion relation $k\left(\mu\right)=\arccos\left(1-\mu\right)$
from Theorem \ref{thm:discrete metric} relating $\protect\spec{\Delta}$
(horizontal) and $\protect\spec H$ (vertical). Each $\mu\in\protect\spec{\Delta}$
corresponds to a point $k\left(\mu\right)^{2}\in\protect\spec H$.
The dispersion relation $k\left(\mu\right)$ is either monotone increasing
or monotone decreasing, depending on the parity of $\left\lfloor \frac{k\left(\mu\right)}{\pi}\right\rfloor $.
\label{fig: dispersion}}
\end{figure}
Let $E\geq0$. Write $E=k^{2}$, and count the total number of square
roots of eigenvalues of $H$ in $\left[0,k\right]$ (with multiplicity).
We write $k=\pi m+r$, with $m\in\N$ and $r\in\left[0,\pi\right)$.
We use Theorem\,\ref{thm:discrete metric} to present the number
of (square roots of) eigenvalues in $\left[0,k\right]$ as the sum
of eigenvalue counts in $\left[0,\pi m\right]$ and in $\left[\pi m,k\right]$.
Noting that $m=\left\lfloor \frac{k}{\pi}\right\rfloor $ and applying
Theorem~\ref{thm:discrete metric} we get
\begin{align}
\#\left\{ \lambda\in\spec H\right. & :\left.\lambda\le k^{2}\right\} -\left\lfloor \frac{k}{\pi}\right\rfloor \left|\E_{G}\right|=\nonumber \\
= & \begin{cases}
\#\set{\mu\in\spec{\Delta}}{\mu\le1-\cos\left(k\right)}, & \left\lfloor \frac{k}{\pi}\right\rfloor \text{ is even,}\\
\#\set{\mu\in\spec{\Delta}}{\mu\ge1-\cos\left(k\right)}, & \left\lfloor \frac{k}{\pi}\right\rfloor \text{ is odd,}
\end{cases}\nonumber \\
= & \begin{cases}
\#\set{\mu\in\spec{\Delta}}{\mu\le1-\cos\left(k\right)}, & \left\lfloor \frac{k}{\pi}\right\rfloor \text{ is even,}\\
\left|\V_{G}\right|-\#\set{\mu\in\spec{\Delta}}{\mu<1-\cos\left(k\right)}, & \left\lfloor \frac{k}{\pi}\right\rfloor \text{ is odd,}
\end{cases}\label{eq:disc-met-2}
\end{align}
where all eigenvalue counts above are with multiplicities (i.e., $\#\left\{ \phantom{\lambda}\right\} $
is considered as element counting of a multi-set). In the first equality
above we need to separate cases according to the parity of $m=\left\lfloor \frac{k}{\pi}\right\rfloor $,
since the dispersion relation $1-\cos(k)$ in (\ref{eq:discrete-metric})
is monotone increasing when $\left\lfloor \frac{k}{\pi}\right\rfloor $
is even and decreasing when $\left\lfloor \frac{k}{\pi}\right\rfloor $
is odd. See Figure~\ref{fig: dispersion}, where the inverse dispersion
relation $\lambda\left(\mu\right)=\arccos\left(1-\mu\right)$ is depicted.

Next, let $\omega\in\Omega$ be in the full measure set for which
Proposition\ \ref{prop:IDS-existence} holds and choose the sequences
of compact graphs $\Gamma^{(n)}:=\pa$ and $G^{(n)}:=\dpa$ as in
(\ref{eq:truncation-IDS}),(\ref{eq:truncation-IDS-2}) and take the
limit $n\rightarrow\infty$ to get the IDS. We perform the computation
only for the case of odd $\left\lfloor \frac{k}{\pi}\right\rfloor $.
The complementary case involves a similar (and slightly simpler) computation.
Using the convergence stated in Proposition\ \ref{prop:IDS-existence}
and applying (\ref{eq:disc-met-2}) we compute:
\begin{align}
N_{\Omega}^{H}\left(k^{2}\right)= & \lim_{n\rightarrow\infty}\frac{\#\set{\lambda\in\spec{\left.H_{\omega}\right|_{\Gamma^{(n)}}}}{\lambda\le k^{2}}}{\left|\Gamma^{(n)}\right|}\nonumber \\
= & \lim_{n\rightarrow\infty}\frac{\#\set{\lambda\in\spec{\left.H_{\omega}\right|_{\Gamma^{(n)}}}}{\lambda\le k^{2}}}{\left|\E_{G^{(n)}}\right|}\nonumber \\
= & \lim_{n\rightarrow\infty}\frac{1}{\left|\E_{G^{(n)}}\right|}\cdot\left\lfloor \frac{k}{\pi}\right\rfloor \text{\ensuremath{\left|\E_{G^{(n)}}\right|}}\nonumber \\
\nonumber \\\  & +\lim_{n\rightarrow\infty}\frac{\left|\V_{G^{(n)}}\right|-\#\set{\mu\in\spec{\left.\Delta\right|_{G^{(n)}}}}{\mu<1-\cos\left(k\right)}}{\left|\E_{G^{(n)}}\right|}\nonumber \\
\nonumber \\= & \left\lfloor \frac{k}{\pi}\right\rfloor +\lim_{n\rightarrow\infty}\frac{\left|\V_{G^{(n)}}\right|}{\left|\E_{G^{(n)}}\right|}\thinspace\frac{\left|\V_{G^{(n)}}\right|-\set{\mu\in\spec{\left.\Delta\right|_{G^{(n)}}}}{\mu<1-\cos\left(k\right)}}{\left|\V_{G^{(n)}}\right|},\label{eq:disc-met-3}
\end{align}
where in the first equality we used that $\Gamma^{(n)}$ has all edge
lengths equal to $1$ and so $\left|\Gamma^{(n)}\right|=\left|\E_{G^{(n)}}\right|$.
The prefactor inside the limit above is
\begin{align}
\lim_{n\rightarrow\infty}\frac{\left|\V_{G^{(n)}}\right|}{\left|\E_{G^{(n)}}\right|}= & \lim_{n\rightarrow\infty}\frac{\sum_{a\in\A}\#_{a}^{n}\left(\omega\right)\left|\V_{G_{a}}\right|}{n+\sum_{a\in\A}\#_{a}^{n}\left(\omega\right)\left|\E_{G_{a}}\right|}\nonumber \\
= & \lim_{n\rightarrow\infty}\frac{\sum_{a\in\A}\frac{\#_{a}^{n}\left(\omega\right)}{n}\left|\V_{G_{a}}\right|}{1+\sum_{a\in\A}\frac{\#_{a}^{n}\left(\omega\right)}{n}\left|\E_{G_{a}}\right|}=\frac{\sum_{a\in\A}\freq a\left|\V_{G_{a}}\right|}{1+\sum_{a\in\A}\freq a\left|\E_{G_{a}}\right|}=C\left(G_{\Omega}\right).\label{eq:Betti-avg}
\end{align}
Substituting this above and recalling that $k=\sqrt{E}$ gives
\begin{align}
N_{\Omega}^{H}\left(E\right)= & \left\lfloor \frac{\sqrt{E}}{\pi}\right\rfloor +C\left(G_{\Omega}\right)\lim_{n\rightarrow\infty}\frac{\left|\V_{G^{(n)}}\right|-\set{\mu\in\spec{\left.\Delta\right|_{G^{(n)}}}}{\mu<1-\cos\left(\sqrt{E}\right)}}{\left|\V_{G^{(n)}}\right|}\nonumber \\
= & \left\lfloor \frac{\sqrt{E}}{\pi}\right\rfloor +C\left(G_{\Omega}\right)\left[1-N_{\Omega}^{\Delta}\left(1-\cos\left(\sqrt{E}\right)\right)\right].\label{eq:final-cases}
\end{align}
Note that by definition, 
\begin{equation}
N_{\Omega}^{\Delta}\left(1-\cos\left(\sqrt{E}\right)\right)=\lim_{n\rightarrow\infty}\frac{1}{\left|\V_{G^{(n)}}\right|}\set{\mu\in\spec{\left.\Delta\right|_{G^{(n)}}}}{\mu\le1-\cos\left(\sqrt{E}\right)}.\label{eq:IDS-def}
\end{equation}
 Nevertheless, the last equality above is justified (even though the
strict inequality $\mu<1-\cos\left(k\right)$ appears), since we assume
that $N_{\Omega}^{H}\left(E\right)$ is continuous at $E$. The distinction
between strict and non-strict inequality in the spectral counting
functions matters only when there is a discontuity in the IDS (see
more on jump discontinuities of the IDS in Section\ \ref{sec:Discontinuities-in-the-IDS}).

For the case when $\left\lfloor \frac{\sqrt{E}}{\pi}\right\rfloor $
is even a similar computation gives $N_{\Omega}^{H}\left(E\right)=\left\lfloor \frac{\sqrt{E}}{\pi}\right\rfloor +C\left(G_{\Omega}\right)\cdot N_{\Omega}^{\Delta}\left(1-\cos\left(\sqrt{E}\right)\right)$.
\end{proof}
\begin{proof}[Proof of Theorem \ref{thm:Discrete-GLT}]
 By Proposition \ref{prop:counting-functions}, 
\begin{equation}
N_{\Omega}^{\Delta}\left(1-\cos\left(\sqrt{E}\right)\right)=\begin{cases}
\frac{1}{C\left(G_{\Omega}\right)}\left(N_{\Omega}^{H}\left(E\right)-\left\lfloor \frac{\sqrt{E}}{\pi}\right\rfloor \right), & \left\lfloor \frac{\sqrt{E}}{\pi}\right\rfloor \text{ is even,}\\
\\1-\frac{1}{C\left(G_{\Omega}\right)}\left(N_{\Omega}^{H}\left(E\right)-\left\lfloor \frac{\sqrt{E}}{\pi}\right\rfloor \right), & \left\lfloor \frac{\sqrt{E}}{\pi}\right\rfloor \text{ is odd.}
\end{cases}\label{eq:ap-relation}
\end{equation}

By Theorem~\ref{thm:GLT}, if $E\notin\spec{H_{\Omega}}$ then 
\begin{equation}
N_{\Omega}^{H}\left(E\right)\in\frac{\S}{\overline{\E}(G_{\Omega})},\label{eq:disc-glt-0}
\end{equation}
where we used $\overline{L}(\Gamma_{\Omega})=\overline{\E}(G_{\Omega})$
which holds since the metric graphs are equilateral with each edge
length equal to $1$. We fix $\omega\in\Omega$ to be in the full
measure set for which $\spec{H_{\omega}}=\spec{H_{\Omega}}$ and $\spec{\Delta_{\omega}}=\spec{\Delta_{\Omega}}$
(see Section~\ref{subsec:Comb-graphs}) and the spectral counting
functions converge to the IDS as in Proposition~\ref{prop:IDS-existence}
and Equation~(\ref{eq:truncation-IDS-2}). From Theorem~\ref{thm:discrete metric}
we conclude that $E$ is inside a spectral gap of $H_{\omega}$ if
and only if $1-\cos(\sqrt{E})$ is inside a spectral gap of $\Delta_{\omega}$.
Therefore, the possible gap labels of $\Delta_{\omega}$ (and hence
of $\spec{\Delta_{\Omega}}$) may be obtained by substituting (\ref{eq:disc-glt-0})
in (\ref{eq:ap-relation}). For this, recall that $\overline{\E}(G_{\Omega})=1+\sum_{a\in\A}\nu_{a}\E(G_{a})$
and that $\nu_{a}\in\S$ for all $a\in\A$ and $\Z\subset\S$ (as
is explained in the end of the proof of Theorem\ \ref{thm:GLT}).
Therefore $\overline{\E}(G_{\Omega})\in\S$ and so for $E\notin\spec{H_{\Omega}}$,
\begin{equation}
N_{\Omega}^{H}\left(E\right)+\Z\in\frac{\S+\Z\thinspace\overline{\E}(G_{\Omega})}{\overline{\E}(G_{\Omega})}\subset\frac{\S}{\overline{\E}(G_{\Omega})},\label{eq:N+Z}
\end{equation}
and using $C\left(G_{\Omega}\right)=\frac{\overline{\V}(G_{\Omega})}{\overline{\E}(G_{\Omega})}$
we get 
\begin{equation}
\frac{1}{C\left(G_{\Omega}\right)}\left(N_{\Omega}^{H}\left(E\right)+\Z\right)\subset\frac{\S}{\overline{\V}(G_{\Omega})}.\label{eq:disc-glt-1}
\end{equation}
Using again that $\nu_{a}\in\S$ for all $a\in\A$, we get $\overline{\V}(G_{\Omega})\subset\S$,
which yields that 
\begin{equation}
1-\frac{1}{C\left(G_{\Omega}\right)}\left(N_{\Omega}^{H}\left(E\right)+\Z\right)\subset\frac{\S}{\overline{\V}(G_{\Omega})}.\label{eq:disc-glt-2}
\end{equation}

From (\ref{eq:disc-glt-1}) and (\ref{eq:disc-glt-2}), both cases
in (\ref{eq:ap-relation}) yield the same gap labels,
\begin{equation}
\mathcal{GL}\left(N_{\Omega}^{\Delta}\right)\subset\frac{\S}{\overline{\V}(G_{\Omega})}\cap[0,1].\label{GL}
\end{equation}
\end{proof}

\section{Discontinuities of the IDS - Proof of Theorem~\ref{thm:IDS-jumps}\label{sec:Discontinuities-in-the-IDS}}

Theorems \ref{thm:GLT} and \ref{thm:Discrete-GLT} provide the set
of all possible gap labels for operators on metric and discrete decorated
$\Z$-graphs. A well-known problem is to find whether all gap labels
predicted by such gap labelling theorems actually occur. This is called
the dry ten Martini problem, originating in a question by Mark Kac
about the almost Mathieu operator \cite{Simon1982}. In this section
we discuss a specific form of obstructions for the appearance of the
predicted gaps. Since the predicted gap labels form a dense set, any
discontinuity of the IDS implies the existence of labels that are
not realized (also known as closed gaps). We illustrate this by completely
analyzing the IDS jumps for metric Sturmian combs (see Example~\ref{exa:Sturmian comb}).
The necessary and sufficient conditions for IDS discontinuities of
these graphs are given in Theorem~\ref{thm: IDS jumps Sturmian combs}.
In addition, the theorem explicitly states all the energies at which
such discontinuities occur and the size of the IDS jump at those energies.
Theorem ~\ref{thm:IDS-jumps} is an immediate corollary of Theorem
~\ref{thm: IDS jumps Sturmian combs}.
\begin{thm}
\label{thm: IDS jumps Sturmian combs} Let $\alpha\in\left(0,1\right)\backslash\Q$,
written as the following infinite continued fraction:
\begin{equation}
\alpha=\frac{1}{c_{1}+\frac{1}{c_{2}+...}}.\label{eq:cf-expansion-2-1}
\end{equation}
Let $\Omega_{\alpha}$ be the corresponding Sturmian subshift. Then
the IDS for the associated family of Sturmian combs $\left(\Gamma_{\omega}\right)_{\omega\in\Omega_{\alpha}}$
has discontinuities if and only if one of the following holds:
\begin{enumerate}
\item $\frac{\ell}{L}=\frac{2m+1}{2n}(c_{1}+1)$ for some $m,n\in\N$.\\
In this case the IDS is discontinuous at $E=\left(\frac{\pi n}{L\left(c_{1}+1\right)}\right)^{2}$,
and the associated jump in the IDS value is
\begin{equation}
\Delta N_{\Omega_{\alpha}}\left(E\right)=\frac{1-c_{1}\alpha}{L+\alpha\ell},\label{eq:bigjump}
\end{equation}
or\\
~
\item $\frac{\ell}{L}=\frac{2m+1}{2n}c_{1}$ for some $m,n\in\N$. \\
In this case, the IDS is discontinuous at $E=\left(\frac{\pi n}{Lc_{1}}\right)^{2}$,
and the associated jump in the IDS value is
\begin{equation}
\Delta N_{\Omega_{\alpha}}\left(E\right)=\frac{(c_{1}+1)\alpha-1}{L+\alpha\ell}.\label{eq:smalljump}
\end{equation}
\end{enumerate}
If both conditions on $\ell/L$ above hold simultaneously, i.e., $\frac{\ell}{L}=\frac{2m_{1}+1}{2n_{1}}(c_{1}+1)=\frac{2m_{2}+1}{2n_{2}}c_{1}$
for $m_{1},n_{1},m_{2},n_{2}\in\N$, then the IDS is discontinuous
at $E=\left(\frac{\pi n_{1}}{L\left(c_{1}+1\right)}\right)^{2}=\left(\frac{\pi n_{2}}{Lc_{1}}\right)^{2}$,
and the associated jump in the IDS value is the sum of (\ref{eq:smalljump})
and (\ref{eq:bigjump}), i.e., 
\begin{equation}
\Delta N_{\Omega_{\alpha}}\left(E\right)=\frac{\alpha}{L+\alpha\ell}.\label{eq: sumjump}
\end{equation}

\end{thm}

\begin{rem*}
Note that if either case in the theorem occurs, it does so for infinitely
many pairs $\left(m,n\right)$, hence the IDS has jumps at infinitely
many energies.
\end{rem*}
We show that the IDS discontinuities are caused by compactly supported
eigenfunctions. A detailed resolution to the dry ten Martini problem
for Sturmian metric graphs is given in \cite{Band}. Two intriguing
recent works \cite{DamEmbFilMei_exp23,SchFraFliGru_prl24} explore
IDS discontinuities in aperiodic discrete graphs, which are also due
to compactly supported eigenfunctions. Some fundamental results on
this phenomenon for random operators on aperiodic discrete graphs
appeared already in \cite{KlaLenSto_cmp03}. Similar phenomena is
observed also in periodic graphs models, as was analyzed for discrete
graphs \cite{Peyerimhoff2021} and metric graphs \cite{Lenz2009}
(see also \cite{PeyTauVes_nano_2017} where continuous models are
confronted with metric and discrete graphs). The most recent work
on the IDS of quantum graphs and their discontinuities appear in \cite{BreLev_arXiv_2026,Levi_prep_2026}
where periodic metric trees are analyzed.

\ 

To prove Theorem\ \ref{thm: IDS jumps Sturmian combs} we need two
lemmas. Lemma\ \ref{lem: compactly supported between teeth} shows
that all the compactly supported eigenfunctions of $\Gamma_{\omega}$
are supported on specific subgraphs. These subgraphs are associated
with particular subwords of $\omega\in\Omega_{\alpha}$ and Lemma\ \ref{lem:Frequencies}
expresses the frequencies of these subwords.
\begin{lem}
\label{lem: compactly supported between teeth} Let $\alpha\in\left(0,1\right)\backslash\Q$,
$\omega\in\Omega_{\alpha}$ and $E\in\R$. A compactly supported $E$-eigenfunction
of the Sturmian comb $\Gamma_{\omega}$ exists if and only if there
exists an $E$-eigenfunction which is supported between two adjacent
teeth in $\Gamma_{\omega}$.
\end{lem}

\begin{proof}
One direction is trivial. For the converse, let $f$ be a compactly
supported solution to $-\frac{d^{2}f}{dx^{2}}=Ef$ on $\Gamma_{\omega}$,
satisfying Neumann-Kirchhoff vertex conditions. Since $f$ is compactly
supported, choose the two farthest teeth on which it is supported,
and get that $f$ must vanish at the base of each of these two teeth
(i.e., the vertex which connects them to the $\Z$-graph). Since $f$
has a vanishing derivative at the other vertex of each of these teeth
(i.e., the boundary vertex), we get that $k\ell=\frac{\pi}{2}+\pi m$
for some $m\in\mathbb{N}$, where $k:=\sqrt{E}$ and $\ell$ is the
tooth length. This implies that $f$ in fact vanishes at the base
of all teeth of the comb. Now, choose a (horizontal) path $\tilde{e}$
between two adjacent teeth $e_{1},e_{2},$ such that $f$ does not
identically vanish on this path. At the bases of these teeth $v_{1},v_{2}$
we have that $f\left(v_{1}\right)=f\left(v_{2}\right)=0$, by the
argument given above. Construct a new eigenfunction $\tilde{f}$ as
follows:

1. At the horizontal path set $\tilde{f}=f$.

2. At the two mentioned teeth $e_{1},e_{2}$, set $\left.\tilde{f}\right|_{e_{i}}\left(x\right)=A_{i}\sin\left(\left(\frac{\frac{\pi}{2}+\pi m}{\ell}\right)x\right)$
for $i\in\left\{ 1,2\right\} $. Choose $A_{1},A_{2}$ so that $\left.\tilde{f}'\right|_{e_{i}}(v_{i})+\left.\tilde{f}'\right|_{\tilde{e}}(v_{i})=0$.

3. Extend $\tilde{f}$ to be identically $0$ everywhere else.

The resulting function is an $E$-eigenfunction supported between
two adjacent teeth of $\Gamma_{\omega}$.
\end{proof}
Towards the next lemma, we define the frequency of a subword. Let
$\Omega_{\alpha}$ be a Sturmian subshift, and $W=W_{0}...W_{k}$
a finite subword over the alphabet $\A=\{0,1\}$. Let $\omega\in\Omega_{\alpha}$.
We denote
\begin{equation}
\nu_{W}:=\lim_{N\rightarrow\infty}\frac{\#\set{n\in\left\{ 0,...,N-1\right\} }{\left.\omega\right|_{\left[n,n+k\right]}=W}}{N},\label{eq:word-freq-1}
\end{equation}
which is the frequency with which the subword $W$ occurs in $\omega$,
and is actually invariant with respect to $\omega\in\Omega$ due to
unique ergodicity (see Subsection~\ref{subsec: Dynamics} and \cite[prop. 4.4]{Baake2013}).
We therefore refer to $\nu_{W}$ as the frequency with which $W$
occurs in the subshift $\Omega_{\alpha}$.
\begin{lem}
\label{lem:Frequencies}Let $\alpha\in\left(0,1\right)\backslash\Q$
with the continued fraction expansion (\ref{eq:cf-expansion-2-1}).
Then there exist only two subwords of the form 
\begin{equation}
W=1\underset{k}{\underbrace{0....0}}1\label{eq:Word}
\end{equation}
which occur in the subshift $\Omega_{\alpha}$:
\end{lem}

\begin{enumerate}
\item A subword $W$ with $k=c_{1}$ zeros, which appear with frequency
$1-c_{1}\alpha$ in $\Omega_{\alpha}$.
\item A subword $W$ with $k=c_{1}-1$ zeros, which appears with frequency
$(c_{1}+1)\alpha-1$.
\end{enumerate}
\begin{proof}
Given a finite word $W$ we consider the following subset of $S^{1}$:
\begin{equation}
I_{W}:=\set{\theta\in S^{1}}{\left.\omega_{\alpha,\theta}\right|_{\left[0,...,\left|W\right|-1\right]}=W},\label{eq:Iw}
\end{equation}
where $\omega_{\alpha,\theta}(n):=\chi_{(1-\alpha,1]}\left(n\alpha+\theta\text{ mod \ensuremath{1}}\right)$
is a Sturmian (infinite) word such that $\omega_{\alpha,\theta}\in\Omega_{\alpha}$.
By \cite[sec. 2.2.3]{Lothaire2002} (see also \cite[sec. 5]{BaaGahMaz_ijc24}),
the frequency of the subword $W$ in $\Omega_{\alpha}$ is equal to
the Lebesgue measure of $I_{W}$. We therefore compute the Lebesgue
measure $I_{W}$ for all admissible subwords of the form $W=10....01$.
We accompany the proof with Figure~\ref{fig:cpt-efun}.

\begin{figure}
\includegraphics[scale=0.7]{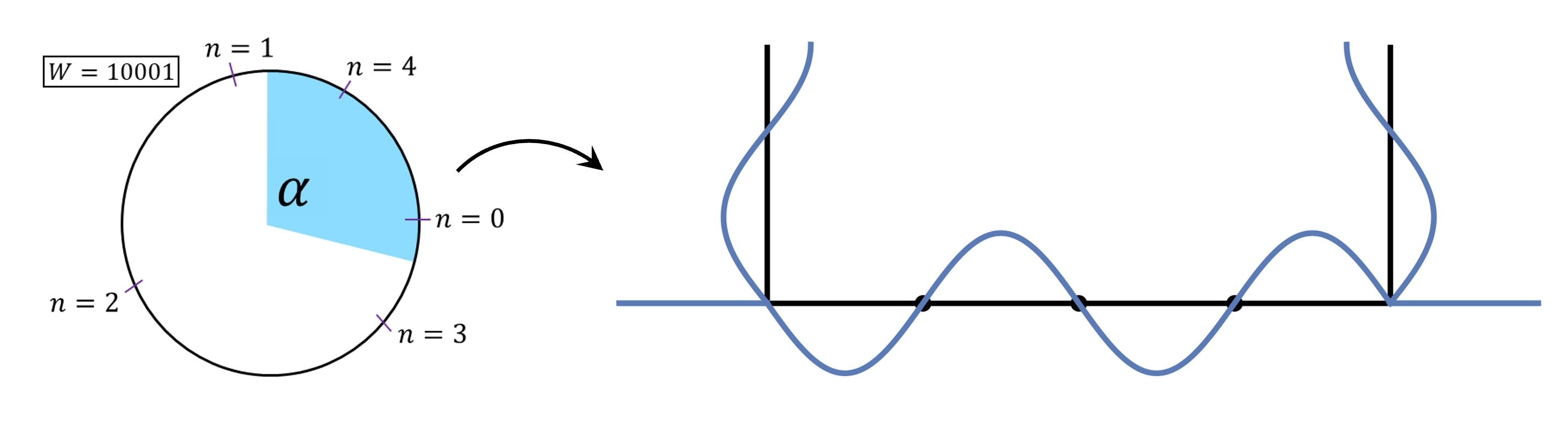}

\caption{Illustration of how subwords of the form $W=10...01$ give rise to
a compactly supported eigenfunction. Here, taking $\alpha\approx0.29$
and the initial angle to be $\theta=1-\alpha+\varepsilon$, the first
subword of length $5$ of the associated Sturmian sequence (\ref{eq:Sturmian-sequence})
gives rise to the compactly supported eigenfunction on the right.
\label{fig:cpt-efun}}
\end{figure}

First, note that 
\begin{equation}
\frac{1}{c_{1}+1}<\alpha<\frac{1}{c_{1}}.\label{eq:alpha-c1-bounds}
\end{equation}
By the definition of the sequence $\omega_{\alpha,\theta}$ we have
that $\omega_{\alpha,\theta}(0)=1$ iff $\theta\in(1-\alpha,1]$.
By (\ref{eq:alpha-c1-bounds}) we get that $n\alpha+\theta\in\left[\frac{n-1}{c_{1}+1}+1,\frac{n}{c_{1}}+1\right)$
for all $\theta\in\left[1-\alpha,1\right)$. In particular we get
that $n\alpha+\theta\mod 1\in\left[0,1-\alpha\right)$, for all $1\leq n\leq c_{1}-1$.
We conclude the argument above by 
\begin{equation}
\omega_{\alpha,\theta}(0)=1\quad\Leftrightarrow\quad\theta\in\left[1-\alpha,1\right)\quad\Leftrightarrow\quad\left.\omega_{\alpha,\theta}\right|_{\left[0,...,c_{1}-1\right]}=1\underset{c_{1}-1}{\underbrace{0....0}}.\label{eq:given-word}
\end{equation}
As we wish that $\omega_{\alpha,\theta}(0)=W(0)=1$, we may assume
the above equivalent conditions and split into two cases:
\begin{enumerate}
\item Assume $\omega_{\alpha,\theta}(c_{1})=0$, which is equivalent to
$c_{1}\alpha+\theta\mod 1\in[0,1-\alpha)$. In addition, from $\theta\in[1-\alpha,1)$
we have $c_{1}\alpha+\theta\in[1+(c_{1}-1)\alpha,1+c_{1}\alpha)$
and so $c_{1}\alpha+\theta\mod 1\in[(c_{1}-1)\alpha,c_{1}\alpha)$.
Intersecting both intervals gives $c_{1}\alpha+\theta\mod 1\in$ $[(c_{1}-1)\alpha,1-\alpha)$.
From here we get $(c_{1}+1)\alpha+\theta\mod 1\in[c_{1}\alpha,1)$,
and this implies $\omega_{\alpha,\theta}(c_{1}+1)=1$. Concluding
we get that in this case
\begin{equation}
\left.\omega_{\alpha,\theta}\right|_{\left[0,...,c_{1}+1\right]}=W=1\underset{c_{1}}{\underbrace{0....0}}1.\label{eq:given-word-1}
\end{equation}
We need also to know the range of $\theta$ for this case, namely
what is $I_{W}$ for the subword $W$ above. In the current case,
we got $c_{1}\alpha+\theta\mod 1\in$ $[(c_{1}-1)\alpha,1-\alpha)$.
This means that $\theta\in[-\alpha,1-(c_{1}+1)\alpha)$. The Lebesgue
measure of this interval is $1-c_{1}\alpha$, which is the frequency
of the word $W$ above.
\item Assume $\omega_{\alpha,\theta}(c_{1})=1$, which is equivalent to
$c_{1}\alpha+\theta\mod 1\in[1-\alpha,1)$. Repeating the arguments
as in the case above we get that $I_{W}=[2-(c_{1}+1)\alpha,1)$ for
$W=1\underset{c_{1}-1}{\underbrace{0....0}}1$. The Lebesgue measure
of this interval is $(c_{1}+1)\alpha-1$, which is the frequency of
that word.
\end{enumerate}
The two cases above exhaust all subwords of the form $W=10....01$
occurring in $\Omega_{\alpha}$.
\end{proof}
\begin{proof}[Proof of Theorem \ref{thm: IDS jumps Sturmian combs}]
 We start the proof by referring to Corollary \ref{cor:freq-jump},
whose hypothesis holds because all finite subwords of a Sturmian subshift
have positive frequency (see beginning of proof of Lemma \ref{lem:Frequencies},
or similar arguments in \cite[sec. 2.2.3]{Lothaire2002} and \cite[sec. 5]{BaaGahMaz_ijc24}).
We conclude from Corollary \ref{cor:freq-jump} that $N_{\Omega_{\alpha}}$
has a jump discontinuity at energy $E$ if and only if $E$ admits
a compactly supported eigenfunction. By Lemma~\ref{lem: compactly supported between teeth}
compactly supported $E$-eigenfunctions exist precisely when there
is an $E$-eigenfunction supported between two adjacent teeth of the
graph, see Figure \ref{fig:cpt-efun}.

Let $f$ be an $E$-eigenfunction which is supported between adjacent
teeth and denote $k:=\sqrt{E}$. The following holds:

(1) By the proof of Lemma\ \ref{lem: compactly supported between teeth},
$f$ vanishes at the base of each tooth, and $k\ell=\frac{\pi}{2}+\pi m$
for some $m\in\mathbb{N}$.

(2) By Lemma~\ref{lem:Frequencies}, the (horizontal) distance between
adjacent teeth in $\Gamma_{\omega}$ is either $c_{1}L$ or $(c_{1}+1)L$.
By (1) above, this implies that $k\left(c_{1}+1\right)L=\pi n$ or
$kc_{1}L=\pi n$ for $n\in\N$. It may be that both equalities hold
for the same value of $k$, but with different $n$ values.

We now examine the two cases in (2). First consider $k\left(c_{1}+1\right)L=\pi n$.
Combining this with the $k\ell=\frac{\pi}{2}+\pi m$, translates into
the following condition:
\begin{equation}
\frac{\ell}{L}=\frac{\left(2m+1\right)\left(c_{1}+1\right)}{2n},\,m,n\in\N,\label{eq:lL2}
\end{equation}
and the corresponding eigenvalue is $E=k^{2}=\left(\frac{\pi n}{L(c_{1}+1)}\right)^{2}$.
This is precisely the case demonstrated in Figure \ref{fig:cpt-efun}
with $n=4$, $m=1$, and $c_{1}=3$.

Similarly, the case $kc_{1}L=\pi n$ translates into the following
condition:

\begin{equation}
\frac{\ell}{L}=\frac{\left(2m+1\right)c_{1}}{2n},\,m,n\in\N,\label{eq:lL1}
\end{equation}
and the corresponding eigenvalue is $E=k^{2}=\left(\frac{\pi n}{Lc_{1}}\right)^{2}$.

These are exactly the two possible conditions on $\ell/L$ and the
corresponding energies in the theorem. It remains to compute the jump
size. Using (\ref{eq:truncation-IDS}), we count the number of compactly
supported eigenfunctions for the finite truncations $\left.H_{\omega}\right|_{[0,N]}$.
For the energy $E=\left(\frac{\pi n}{L\left(c_{1}+1\right)}\right)^{2}$,
we need to consider the subword $W=1\underset{c_{1}}{\underbrace{0....0}}1$
and the eigenfunctions supported on the corresponding subgraphs. These
eigenfunctions are linearly independent and so
\begin{equation}
\dim\ker\left(\left.H_{\omega}\right|_{[0,N]}-E\right)=\#\set{j\in\left\{ 0,...,N-c_{1}-1\right\} }{\left.\omega\right|_{[j,j+c_{1}+1]}=W}.\label{eq:dimker}
\end{equation}
 The jump in the IDS at $E$ is given by 
\begin{align}
 & \Delta N_{\Omega_{\alpha}}\left(E\right)\nonumber \\
 & =\lim_{N\rightarrow\infty}\frac{\#\set{\lambda\in\spec{\left.H_{\alpha}\right|_{[0,N]}}}{\lambda\leq E}-\#\set{\lambda\in\spec{\left.H_{\alpha}\right|_{[0,N]}}}{\lambda<E}}{\left|\left.\Gamma_{\alpha}\right|_{[0,N]}\right|}\nonumber \\
 & =\lim_{N\rightarrow\infty}\frac{\dim\ker\left(\left.H_{\alpha}\right|_{[0,N]}-E\right)}{\left|\left.\Gamma_{\alpha}\right|_{[0,N]}\right|}\nonumber \\
 & =\lim_{N\rightarrow\infty}\frac{\#\set{j\in\left\{ 0,...,N-c_{1}-1\right\} }{\left.\omega_{\alpha}\right|_{[j,j+c_{1}+1]}=W}}{\left|\left.\Gamma_{\alpha}\right|_{[0,N]}\right|}\nonumber \\
 & =\lim_{N\rightarrow\infty}\frac{\left(N-c_{1}\right)\nu_{W}}{NL+\alpha N\ell}=\frac{1-c_{1}\alpha}{L+\alpha\ell},\label{eq:jump-computation}
\end{align}
where the last line is obtained by Lemma \ref{lem:Frequencies} according
to which $\nu_{W}=1-c_{1}\alpha$ (see also the definition of word
frequency, (\ref{eq:word-freq-1})).

Repeating the same computation for the energy $E=\left(\frac{\pi n}{Lc_{1}}\right)^{2}$
whose eigenfunctions correspond to the subword $W=1\underset{c_{1}-1}{\underbrace{0....0}}1$.
The only change which is required in the computation is in using the
word frequency which is now $\nu_{W}=(c_{1}+1)\alpha-1$, and we get
\begin{align}
\Delta N_{\Omega_{\alpha}}\left(E\right) & =\frac{(c_{1}+1)\alpha-1}{L+\alpha\cdot\ell}.\label{eq:jump-computation-2}
\end{align}
It may happen that both (\ref{eq:lL2}) and (\ref{eq:lL1}) hold (but
for different $n,m$ values). Namely, 
\begin{equation}
\frac{\ell}{L}=\frac{\left(2m_{1}+1\right)\left(c_{1}+1\right)}{2n_{1}}=\frac{\left(2m_{2}+1\right)c_{1}}{2n_{2}}\label{eq:l/L}
\end{equation}
 for some $m_{1},n_{1},m_{2},n_{2}\in\N$. The corresponding energy
is then $E=\left(\frac{\pi n_{1}}{L(c_{1}+1)}\right)^{2}=\left(\frac{\pi n_{2}}{Lc_{1}}\right)^{2}$
and the associated eigenfunctions are supported on subgraphs corresponding
to both subwords $1\underset{c_{1}}{\underbrace{0\cdot....\cdot0}}1$
and $1\underset{c_{1}-1}{\underbrace{0\cdot....\cdot0}}1$. These
eigenfunctions are linearly independent and so the dimensions of the
corresponding eigenspaces sum up (and the same holds for the frequencies).
Therefore, the IDS jump at such energies is the sum of (\ref{eq:jump-computation})
and (\ref{eq:jump-computation-2}),
\begin{equation}
\Delta N_{\Omega_{\alpha}}\left(E\right)=\frac{1-c_{1}\alpha}{L+\alpha\cdot\ell}+\frac{(c_{1}+1)\alpha-1}{L+\alpha\cdot\ell}=\frac{\alpha}{L+\alpha\cdot\ell}.\label{eq:DN}
\end{equation}
\end{proof}

\appendix

\section{Proof of Proposition\ \ref{prop:IDS-existence}\label{sec:IDS-exists}}

In this appendix we prove Proposition \ \ref{prop:IDS-existence},
namely that the IDS for metric decorated $\Z$-graphs is well-defined
and given by the limit of the spectral counting functions. The discrete
case is analogous and omitted.

In the following, we denote by $\mathcal{F}$ the set of finite subsets
of $\mathbb{Z}$. For any subset $Q\in\mathcal{F}$, let $\left.H_{\omega}\right|_{Q}$
represent the restriction of the operator $H_{\omega}$ to the compact
subgraph $\left.\Gamma_{\omega}\right|_{Q}$ of the decorated $\Z$-graph
$\Gamma_{\omega}$ induced by $Q$ (as in Subsection \ref{subsec:IDS}).
We impose the Dirichlet condition at the boundary vertices of $\left.\Gamma_{\omega}\right|_{Q}$
where the decorated $\Z$-graph $\Gamma_{\omega}$ is truncated, although
other self-adjoint boundary conditions would yield the same results.
We prove the following Pastur--Shubin-type trace formula, which gives
Proposition\ \ref{prop:IDS-existence} as an immediate corollary:
\begin{prop}
\label{prop:PS-trace-formula}Let $\left(\Omega,T\right)$ be a uniquely
ergodic subshift. Denote by $\mu$ the unique shift-invariant probability
measure on $\Omega$. For almost every $\omega\in\Omega$, the sequence
of normalized counting functions $N_{\omega}^{(n)}\left(E\right)$
in (\ref{eq:truncation-IDS}) converges uniformly to a limiting function
$N_{\Omega}\left(E\right)$. For an arbitrary finite $Q\in\mathcal{F}$,
the function $N_{\Omega}$ can be expressed as
\begin{equation}
N_{\Omega}\left(E\right)=\frac{1}{\left|Q\right|\overline{L}\left(\Gamma_{\Omega}\right)}\int_{\Omega}tr\left[\chi_{\left.\Gamma_{\omega}\right|_{Q}}\chi_{(-\infty,E]}\left(H_{\omega}\right)\right]\rmd\mu\left(\omega\right),\label{eq:trace}
\end{equation}
where $\overline{L}\left(\Gamma_{\Omega}\right)$ denotes the average
metric length, as defined in (\ref{eq:norm-length}).
\end{prop}

Our proof relies on an adaptation of the method presented in \cite{GruLenVes_jfa07},
which utilizes an ergodic theorem proven in \cite{LenMulVes_pos08}
(see also \cite{GruLenVes_incol08}). Notably, the proof can be generalized
to many other graph families, including graphs with random potentials
and vertex conditions, higher dimensional decorated graphs (i.e. $\Z^{d}$
with $d>1$), and tiling graphs, as studied in \cite{BanSof_prep24a}.

\subsection{Background and definitions}

We start by introducing essential definitions, and refer to \cite{LenMulVes_pos08}
for more details.

We denote the spectral counting function (and normalized spectral
counting function) for $H_{\omega}^{Q}$ by

\begin{align}
 & n_{\omega}^{Q}\left(E\right):=\#\left\{ \lambda\in\spec{\left.H_{\omega}\right|_{Q}}:\lambda\leq E\right\} ,\label{eq:nQ}\\
 & N_{\omega}^{Q}\left(E\right):=\frac{1}{\left|\left.\Gamma_{\omega}\right|_{Q}\right|}n_{\omega}^{Q}\left(E\right).\label{eq:NQ}
\end{align}
To decouple the graph into its decorations, we further introduce the
operator $\left.H_{\omega,D}\right|_{Q}$, obtained by imposing Dirichlet
conditions at the centers of all edges of the horizontal path. The
corresponding spectral counting function is
\begin{equation}
n_{\omega,D}^{Q}\left(E\right):=\#\left\{ \lambda\in\spec{\left.H_{\omega,D}\right|_{Q}}:\lambda\leq E\right\} .\label{eq:nQD}
\end{equation}
To simplify notation, let $n_{D}^{a}\left(E\right)$ represent the
counting function for the operator $n_{\omega,D}^{Q}$ when $Q=\left\{ a\right\} $
for $a\in\A$. The overall counting function for $H_{\omega,D}^{Q}$
can then be written as:
\begin{equation}
n_{\omega,D}^{Q}\left(E\right)=\sum_{a\in\mathcal{A}}\#_{a}^{Q}\left(\omega\right)n_{D}^{a}(E),\label{eq:nQD2}
\end{equation}
 where $\#_{a}^{Q}\left(\omega\right)$ is the number of occurrences
of the letter $a$ in the subword $\omega|_{Q}$ (extending the definition
of the letter counting function from (\ref{eq:counting})). With the
above, we define the \textit{spectral shift} function
\begin{equation}
\xi_{\omega}^{Q}\left(E\right):=n_{\omega}^{Q}\left(E\right)-n_{\omega,D}^{Q}\left(E\right)=n_{\omega}^{Q}\left(E\right)-\sum_{a\in\mathcal{A}}\#_{a}^{Q}\left(\omega\right)n_{D}^{a}(E).\label{eq:S-Shift}
\end{equation}

Lastly, we provide a few definitions which are required for the proofs.
\begin{defn}
A \textit{van Hove} sequence is a sequence $\left(Q_{j}\right)_{j\in\mathbb{N}}\subset\mathcal{F}$
such that
\begin{equation}
\lim_{j\rightarrow\infty}\frac{\left|\partial Q_{j}\right|}{\left|Q_{j}\right|}=0,\label{eq:Hove}
\end{equation}
where the boundary $\partial Q$ is defined as
\begin{equation}
\partial Q:=\set{n\in Q}{n+1\notin Q\text{ or }n-1\notin Q}.\label{eq:bdry}
\end{equation}
\end{defn}

\begin{defn}
A function $b:\mathcal{F}\rightarrow[0,\infty)$ is called a \textit{boundary
term} if
\end{defn}

\begin{enumerate}
\item $b\left(Q\right)=b\left(m+Q\right)$ for all $m\in\mathbb{Z}$ and
$Q\in\mathcal{F}$,
\item there exists $D>0$ such that $b\left(Q\right)\leq D\left|Q\right|$
for all $Q\in\mathcal{F}$,
\item for any van Hove sequence $\left(Q_{j}\right)_{j\in\N}$, the following
holds:
\begin{equation}
\lim_{j\rightarrow\infty}\frac{b\left(Q_{j}\right)}{\left|Q_{j}\right|}=0.\label{eq:Boundary-map}
\end{equation}
\end{enumerate}
\begin{defn}
~
\begin{enumerate}
\item Let $X$ be a Banach space. $F:\mathcal{F}\rightarrow X$ is called
almost-additive if there exists a boundary term $b$ such that
\begin{equation}
\left\Vert F\left(\cup_{k=1}^{l}Q_{k}\right)-\sum_{k=1}^{l}F\left(Q_{k}\right)\right\Vert \leq\sum_{k=1}^{l}b\left(Q_{k}\right)\label{eq:AAfunction}
\end{equation}
 for all $l\in\mathbb{N}$ and pairwise disjoint sets $Q_{k}$.
\item For a subshift element $\omega\in\Omega$, $F$ is said to be $\omega$-equivariant
if $F\left(Q\right)$ depends only on the local pattern of $\omega$
at $Q$, i.e.,
\begin{equation}
F\left(Q\right)=F\left(m+Q\right),\label{eq:equivariant}
\end{equation}
 whenever $m\in\mathbb{Z}$ and $Q$ obeys $\omega|_{m+Q}=\omega|_{Q}$.
\item $F$ is said to be bounded if there exists $C>0$ such that
\begin{equation}
\left\Vert F\left(Q\right)\right\Vert \leq C\left|Q\right|.\label{eq:bdd-function}
\end{equation}
\end{enumerate}
\end{defn}

\subsection{Proving the main result}

The following paraphrase on the ergodic theorem \cite[thm. 1]{LenMulVes_pos08}
is a main key to the proof of Proposition \ref{prop:PS-trace-formula}:
\begin{thm}
\label{thm:Ergodic}Let $\left(\Omega,T\right)$ be a uniquely ergodic
subshift over $\mathcal{A}$, and let $\omega\in\Omega$. Let $\left(X,\left\Vert \cdot\right\Vert \right)$
be a Banach space, and let $\left(Q_{j}\right)_{j\in\mathbb{N}}$
be a van Hove sequence. Suppose that $F:\mathcal{F}\rightarrow X$
is an $\omega$-equivariant, almost-additive bounded function. Then
the following limit exists:
\begin{equation}
\overline{F}:=\lim_{j\rightarrow\infty}\frac{F\left(Q_{j}\right)}{\left|Q_{j}\right|}.\label{eq:ergodic}
\end{equation}
\end{thm}

\begin{rem*}
\cite[thm. 1]{LenMulVes_pos08} also assumes existence of all subword
frequencies, which here follows from unique ergodicity (see \cite[prop. 4.4]{Baake2013},
\cite{Oxtoby1952}).
\end{rem*}
The following lemma provides the function $F$ on which Theorem\ \ref{thm:Ergodic}
is applied.
\begin{lem}
\label{lem:Spectral shift}On the Banach space $\left(X,\left\Vert \cdot\right\Vert _{\infty}\right)$
of right-continuous bounded functions, define the function
\begin{align}
 & F:\mathcal{F}\rightarrow X,\label{eq:ss-ergodic}\\
 & \left(F\left(Q\right)\right)\left(E\right)=\frac{\xi_{\omega}^{Q}\left(E\right)}{\overline{L}\left(\Gamma_{\Omega}\right)},\label{eq:ss-ergodic-2}
\end{align}
where $\xi_{\omega}^{Q}$ is the spectral shift function (\ref{eq:S-Shift}).
Then $F$ is $\omega$-equivariant, bounded, and almost-additive.
\end{lem}

The proof is similar to \cite[lem. 22]{GruLenVes_jfa07}. Boundedness
follows since $H_{\omega}^{Q}$ and $H_{\omega,D}^{Q}$ differ by
a finite rank perturbation. Similarly, almost-additivity holds since
the disjoint decomposition $Q=\sqcup_{k=1}^{l}Q_{k}$ yields finite
rank perturbations between the associated operators.

The proof of Proposition \ref{prop:PS-trace-formula} now follows,
using conceptually the same arguments as in \cite[thm. 3]{GruLenVes_jfa07}.
\begin{proof}[Proof of Proposition \ref{prop:PS-trace-formula}]
 By Lemma \ref{lem:Spectral shift}, the function $\left(F\left(Q\right)\right)\left(E\right)=\xi_{\omega}^{Q}\left(E\right)/\overline{L}\left(\Gamma_{\Omega}\right)$
is $\omega$-equivariant, almost-additive and bounded. Applying (\ref{eq:S-Shift})
along a van Hove sequence $\left(Q_{j}\right)_{j\in\N}$, we get for
all $j\in\N$
\begin{equation}
n_{\omega}^{Q_{j}}\left(E\right)=\xi_{\omega}^{Q_{j}}\left(E\right)+\sum_{a\in\mathcal{A}}\#_{a}^{Q_{j}}\left(\omega\right)n_{D}^{a}\left(E\right).\label{eq:nQD3}
\end{equation}

Dividing both sides by $\left|\Gamma_{\omega}^{Q_{j}}\right|$ and
taking the limit $j\rightarrow\infty$, we obtain using (\ref{eq:truncation-IDS}):
\begin{align}
N_{\omega}\left(E\right) & =\lim_{j\rightarrow\infty}\frac{n_{\omega}^{Q_{j}}\left(E\right)}{\left|\Gamma_{\omega}^{Q_{j}}\right|}=_{\text{(\ref{eq:norm-length-2})}}\frac{1}{\overline{L}\left(\Gamma_{\Omega}\right)}\lim_{j\rightarrow\infty}\frac{n_{\omega}^{Q_{j}}\left(E\right)}{\left|Q_{j}\right|}\nonumber \\
 & =\frac{1}{\overline{L}\left(\Gamma_{\Omega}\right)}\lim_{j\rightarrow\infty}\left(\frac{\xi_{\omega}^{Q_{j}}\left(E\right)}{\left|Q_{j}\right|}+\frac{1}{\left|Q_{j}\right|}\sum_{a\in\mathcal{A}}\#_{a}^{Q_{j}}\left(\omega\right)n_{D}^{a}(E)\right)\nonumber \\
 & =\frac{1}{\overline{L}\left(\Gamma_{\Omega}\right)}\lim_{j\rightarrow\infty}\frac{\xi_{\omega}^{Q_{j}}\left(E\right)}{\left|Q_{j}\right|}+\frac{1}{\overline{L}\left(\Gamma_{\Omega}\right)}\sum_{a\in\mathcal{A}}\freq an_{D}^{a}\left(E\right),\label{eq:ergodic-IDS}
\end{align}
and the limit exists by Theorem\ \ref{thm:Ergodic}. The convergence
is uniform, since the Banach-space norm in Lemma\ \ref{lem:Spectral shift}
is $\left\Vert \cdot\right\Vert _{\infty}$.

We now prove formula (\ref{eq:trace}), beginning with the $Q$-independence
of the right-hand side. This is clearly true if $Q=\{m\}$, i.e.,
$Q$ is a singleton. This follows from the translation invariance
of the ergodic measure $\mu$, since for all $m\in\Z$:
\begin{align}
 & \int_{\Omega}tr\left[\chi_{\Gamma_{\omega}^{\left\{ m\right\} }}\chi_{(-\infty,E]}\left(H_{\omega}\right)\right]d\mu\left(\omega\right)\nonumber \\
 & =\int_{\Omega}tr\left[\chi_{\Gamma_{T\omega}^{\left\{ m\right\} }}\chi_{(-\infty,E]}\left(H_{T\omega}\right)\right]d\mu\left(\omega\right)\nonumber \\
 & =_{\left(\ref{eq:covariant}\right)}\int_{\Omega}tr\left[\chi_{\Gamma_{T\omega}^{\left\{ m\right\} }}\chi_{(-\infty,E]}\left(T^{-1}H_{\omega}T\right)\right]d\mu\left(\omega\right)\nonumber \\
 & =\int_{\Omega}tr\left[T^{-1}\chi_{\Gamma_{T\omega}^{\left\{ m\right\} }}\chi_{(-\infty,E]}\left(H_{\omega}\right)T\right]d\mu\left(\omega\right)\nonumber \\
 & =_{\left(*\right)}\int_{\Omega}tr\left[\chi_{\Gamma_{T\omega}^{\left\{ m\right\} }}\chi_{(-\infty,E]}\left(H_{\omega}\right)\right]d\mu\left(\omega\right)\nonumber \\
 & =\int_{\Omega}tr\left[\chi_{\Gamma_{\omega}^{\left\{ m+1\right\} }}\chi_{(-\infty,E]}\left(H_{\omega}\right)\right]d\mu\left(\omega\right),\label{eq:point-invariance}
\end{align}
where $\left(*\right)$ follows from the cyclic property of the trace.
For arbitrary $Q$, the claim follows by writing $\frac{1}{\left|Q\right|}\chi_{\Gamma_{\omega}^{Q}}=\frac{1}{\left|Q\right|}\sum_{m\in Q}\chi_{\Gamma_{\omega}^{\left\{ m\right\} }}$.

To prove (\ref{eq:trace}), we first show that
\begin{align}
\lim_{j\rightarrow\infty}\left[\frac{1}{\left|Q_{j}\right|\overline{L}\left(\Gamma_{\Omega}\right)}tr\left\{ \chi_{Q_{j}}f_{z}\left(H_{\omega}\right)\right\} -\frac{1}{\left|\left.\Gamma_{\omega}\right|_{Q_{j}}\right|}tr\left\{ f_{z}\left(\left.H_{\omega}\right|_{Q_{j}}\right)\right\} \right] & =0,\label{eq:trace-id}
\end{align}
for all $f_{z}$ of the form $f_{z}\left(t\right)=\left(t-z\right)^{-1}$
for $z\in\C\backslash\R$, where we use $\chi_{Q_{j}}$ as an abbreviated
notation for $\chi_{\left.\Gamma_{\omega}\right|_{Q_{j}}}$.

For a given box $Q=Q_{j}$, the graph $\Gamma_{\omega}$ naturally
splits into two components, $\left.\Gamma_{\omega}\right|_{Q}$ and
$\left.\Gamma_{\omega}\right|_{\Z\backslash Q}$. In this case, the
operators $H_{\omega}$ and $\left.H_{\omega}\right|_{Q}\oplus\left.H_{\omega}\right|_{\Z\backslash Q}$
only differ by the boundary conditions imposed at the set $\partial Q$.
We consider the operator
\begin{equation}
D:=f_{z}\left(H_{\omega}\right)-f_{z}\left(\left.H_{\omega}\right|_{Q}\oplus\left.H_{\omega}\right|_{\Z\backslash Q}\right).\label{eq:Res-D}
\end{equation}
As it is a difference of two self-adjoint operator and $z\notin\R$,
we get $\left\Vert D\right\Vert \leq2\left|Im\left(z\right)\right|^{-1}$.
In addition, by the second resolvent identity we get $\mathop{rank}D\leq\left|\partial Q\right|$.
Both bounds imply $\left|\mathop{tr}D\right|\leq2\left|\partial Q\right|\left|Im\left(z\right)\right|^{-1}$.
We use this bound to get
\begin{align}
\lim_{j\to\infty} & \left|\frac{1}{|Q_{j}|\overline{L}(\Gamma_{\Omega})}tr\!\left\{ \chi_{Q_{j}}f_{z}(H_{\omega})\right\} -\frac{1}{|\Gamma_{\omega}|_{Q_{j}}|}tr\!\left\{ f_{z}\!\left(H_{\omega}|_{Q_{j}}\right)\right\} \right|\nonumber \\[0.5em]
 & \qquad\qquad=\lim_{j\to\infty}\frac{1}{|Q_{j}|\overline{L}(\Gamma_{\Omega})}\left|tr\!\left\{ \chi_{Q_{j}}f_{z}(H_{\omega})-f_{z}(H_{\omega}|_{Q_{j}})\right\} \right|\nonumber \\[0.5em]
 & \qquad\qquad=\lim_{j\to\infty}\frac{1}{|Q_{j}|\overline{L}(\Gamma_{\Omega})}\left|tr\!\left\{ \chi_{Q_{j}}\!\left(f_{z}(H_{\omega})-f_{z}\!\left(H_{\omega}|_{Q_{j}}\oplus H_{\omega}|_{\mathbb{Z}\setminus Q_{j}}\right)\right)\right\} \right|\nonumber \\[0.5em]
 & \qquad\qquad\le\frac{2}{\overline{L}(\Gamma_{\Omega})\,|Im\left(z\right)|}\lim_{j\to\infty}\frac{|\partial Q_{j}|}{|Q_{j}|}=0.\label{eq:trace-bound}
\end{align}
where in the last equality we used that $Q_{j}$ is van Hove.

A Stone-Weierstrass argument then upgrades $(\ref{eq:trace-id})$
to indicator functions $\chi_{(-\infty,E]}$. Applying this to $Q_{j}=\left[0,j\right]$
and recalling that the normalized spectral functions were defined
as $N_{\omega}^{(j)}(E)=\frac{1}{\left|\left.\Gamma_{\omega}\right|_{Q_{j}}\right|}\mathop{tr}\left\{ \chi_{(-\infty,E]}\left(\left.H_{\omega}\right|_{Q_{j}}\right)\right\} $
we get
\begin{align}
\int_{\Omega}\lim_{j\rightarrow\infty}N_{\omega}^{(j)}\left(E\right)\ \rmd\mu\left(\omega\right) & =\int_{\Omega}\lim_{j\rightarrow\infty}\frac{1}{\left|\left.\Gamma_{\omega}\right|_{Q_{j}}\right|}\mathop{tr}\left\{ \chi_{(-\infty,E]}\left(\left.H_{\omega}\right|_{Q_{j}}\right)\right\} \ \rmd\mu\left(\omega\right)\nonumber \\
 & =\int_{\omega\in\Omega}\lim_{j\rightarrow\infty}\frac{1}{\left|Q_{j}\right|\overline{L}\left(\Gamma_{\Omega}\right)}\mathop{tr}\left\{ \chi_{\left.\Gamma_{\omega}\right|_{Q_{j}}}\chi_{(-\infty,E]}\left(H_{\omega}\right)\right\} \ \rmd\mu\left(\omega\right)\nonumber \\
 & =\frac{1}{\left|Q\right|\overline{L}\left(\Gamma_{\Omega}\right)}\int_{\omega\in\Omega}\mathop{tr}\left\{ \chi_{\left.\Gamma_{\omega}\right|_{Q}}\chi_{(-\infty,E]}\left(H_{\omega}\right)\right\} \ \rmd\mu\left(\omega\right),\label{eq:PS-trace}
\end{align}
where the second equality follows due to ($\ref{eq:trace-id}$), and
the third equality follows from the $Q$-independence of the trace.
This completes the proof.
\end{proof}
\begin{cor}
\label{cor:freq-jump}Assume that the frequencies of all finite subwords
in $\Omega$ are positive. The IDS $N_{\Omega}$ has a jump discontinuity
at $E\in\R$ if and only if $E$ admits a compactly supported eigenfunction.
\end{cor}

The proof follows the same arguments as in \cite[thm. 2]{KlaLenSto_cmp03}
and \cite[cor. 7]{GruLenVes_jfa07}: a jump discontinuity of the IDS
is equivalent to the dimension of the eigenspace of $E\in\spec{\Ha}$
growing ``sufficiently quickly'' in $n$, which then allows one
to construct compactly supported eigenfunctions for $H_{\omega}$.

\section{\label{sec:Proof-of-counting-lemma-1}Proof of Lemma \ref{lem:Counting-lemma}}

This appendix proves Lemma\ \ref{lem:Counting-lemma}, which compares
the spectral counting functions for the Kirchhoff Laplacian on the
graph $\Gamma_{\omega}\left(t\right)$ with those of some of its subgraphs.

The proof relies on continuously interpolating between the full graph
$\Gamma_{\omega}\left(t\right)$ and the disjoint union of the corresponding
subgraphs of $\Gamma_{\omega}\left(t\right)$. This is done using
a one-parameter family of operators $\left(H_{\tau}\right)_{\tau\in\left[0,\pi\right]}$.
This operator family transitions between the full graph and the split
graph while keeping $\fwe$ an eigenfunction for all the graphs in
the transition process, allowing to relate the spectral counting functions
of the graphs.

\begin{figure}
\includegraphics[scale=0.8]{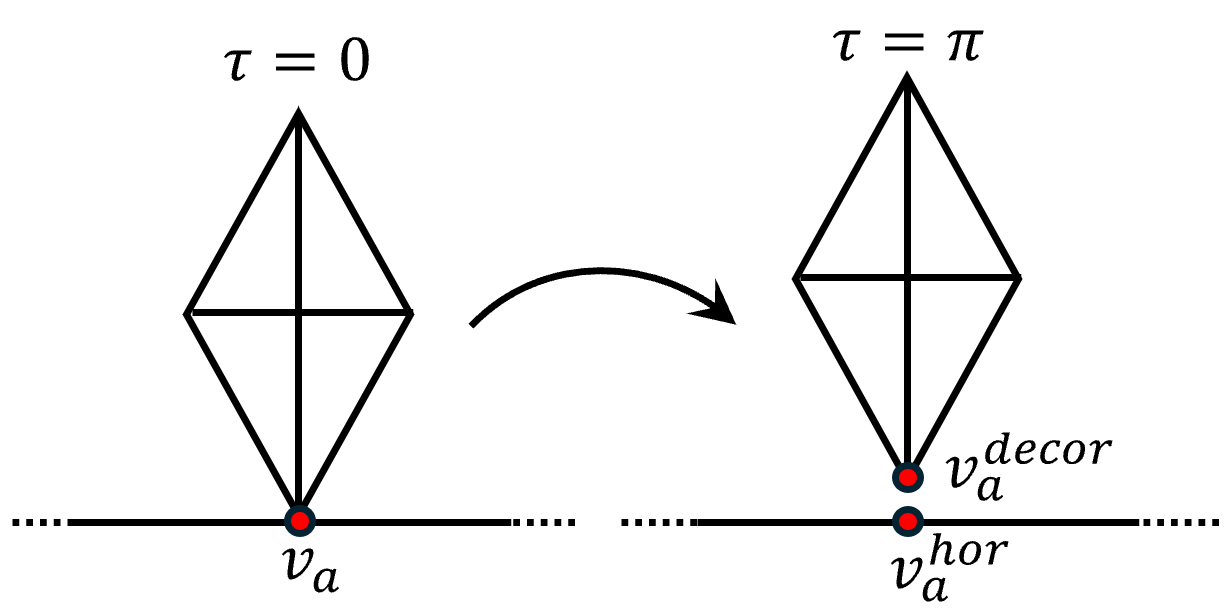}

\caption{The family $\left(H_{\tau}\right)_{\tau\in\left[0,\pi\right]}$ which
continuously disconnects the graph $\Gamma_{\omega}^{t}$ into two
subgraphs as $\tau\rightarrow\pi$. \label{fig: cutgraph-1}}
\end{figure}

To describe the construction, we start by fixing some $E\notin\spec{H_{\omega}}$
such that for each $a\in\A$, it holds that $E$ does not belong to
the spectrum of $\Gamma_{a}$ with a Dirichlet condition at the base
vertex $v_{a}$ of $\Gamma_{a}$ (as in the statement of Lemma~\ref{lem:Counting-lemma}).
Take on $\Gamma_{\omega}$ the unique solution $f_{\omega,E}$, as
is described in step one of the proof of Theorem~\ref{thm:GLT}.
Consider a compact truncation $\Gamma_{\omega}\left(t\right)$ of
the infinite graph $\Gamma_{\omega}$ (see (\ref{eq: truncated-graph})).
The vertices at which the cut is made are $s_{\omega}(t)\cup o(\Gamma_{\omega})$
(see (\ref{eq: spheres def}), (\ref{eq: truncated-graph}), and Figure
\ref{fig: HorizontalCovering}). For each vertex $u\in s_{\omega}(t)\cup o(\Gamma_{\omega})$
we denote
\begin{equation}
\alpha(u):=\frac{\sum_{e\in\E_{u}\cap\Gamma_{\omega}(t)}\left.f_{\omega,E}'\right|_{e}(u)}{f_{\omega,E}(u)},\label{eq:a-u}
\end{equation}
which is considered to be the Robin parameter of $f_{\omega,E}$ at
$u$ when restricted to $\Gamma_{\omega}(t)$. We consider an arbitrary
decoration $\Gamma_{a}$ which is attached to $\Gamma_{\omega}(t)$.
We can split the graph $\Gamma_{\omega}\left(t\right)$ into two subgraphs:
$\Gamma_{a}$ and $\Gamma_{\omega}\left(t\right)\backslash\Gamma_{a}$,
each containing a respective copy of the base vertex $v_{a}$ of the
decoration. We label these copies $v_{a}^{\mathrm{dec}},v_{a}^{\mathrm{hor}}$,
as in Figure \ref{fig: cutgraph-1}.

We describe a family of operators $\left(H_{\tau}\right)_{\tau\in\left[0,\pi\right]}$
on $\Gamma_{\omega}(t)$ which is now considered as a disjoint union
of $\Gamma_{a}$ and $\Gamma_{\omega}(t)\backslash\Gamma_{a}$. Each
operator $H_{\tau}$ acts as the Laplacian on each edge; at the vertices
$s_{\omega}(t)\cup o(\Gamma_{\omega})$ and $v_{a}^{\mathrm{dec}},v_{a}^{\mathrm{hor}}$
it satisfies continuity conditions (for this purpose $v_{a}^{\mathrm{dec}}$
and $v_{a}^{\mathrm{hor}}$ are considered as separate vertices) and
also the following vertex conditions: 
\begin{align}
 & \sum_{e\in\E_{u}}f'|_{e}\left(u\right)=\alpha(u)f(u),\quad\forall u\in s_{\omega}(t)\cup o(\Gamma_{\omega}),\label{eq: operator-family-VC-1-1}\\
 & \sum_{e\sim v_{a}^{\mathrm{dec}}}f_{e}'\left(v_{a}^{\mathrm{dec}}\right)=m_{a}(E)f\left(v_{a}^{\mathrm{dec}}\right)+\cot\left(\tau/2\right)\left(f\left(v_{a}^{\mathrm{hor}}\right)-f\left(v_{a}^{\mathrm{dec}}\right)\right),\label{eq: operator-family-VC-2-1}\\
 & \sum_{e\sim v_{a}^{\mathrm{hor}}}f_{e}'\left(v_{a}^{\mathrm{hor}}\right)=-m_{a}(E)f\left(v_{a}^{\mathrm{hor}}\right)-\cot\left(\tau/2\right)\left(f\left(v_{a}^{\mathrm{hor}}\right)-f\left(v_{a}^{\mathrm{dec}}\right)\right),\label{eq: operator-family-VC-3-1}
\end{align}
where the Robin parameter $m_{a}\left(E\right)$ is a fixed number
determined from $f_{\omega,E}$ restricted to $\Gamma_{a}$ as in
(\ref{eq:m-func}), and the Kirchhoff conditions is imposed at all
other vertices of $\Gamma_{\omega}(t)$. At $\tau=0$ we also add
the requirement $f(v_{a}^{\mathrm{dec}})=f(v_{a}^{\mathrm{hor}})$.
At $\tau=\pi$, the graph $\Gamma_{\omega}\left(t\right)$ is effectively
split at the vertex $v_{a}$, with the Robin conditions imposed at
both $v_{a}^{\mathrm{dec}}$ and $v_{a}^{\mathrm{hor}}$ (but with
opposite signs of the coupling coefficient).

One can verify that $\left.f_{\omega,E}\right|_{\Gamma_{\omega}(t)}$
satisfies the vertex conditions (\ref{eq: operator-family-VC-1-1}),(\ref{eq: operator-family-VC-2-1}),(\ref{eq: operator-family-VC-3-1})
for all $\tau\neq0$ and so it is an eigenfunction of $H_{\tau}$
for all $\tau\neq0$. At $\tau=0$, the additional condition $f(v_{a}^{\mathrm{dec}})=f(v_{a}^{\mathrm{hor}})$,
together with (\ref{eq: operator-family-VC-2-1}),(\ref{eq: operator-family-VC-3-1})
simplifies to Kirchhoff, and therefore $\left.f_{\omega,E}\right|_{\Gamma_{\omega}(t)}$
is an eigenfunction of $H_{0}$ as well. As a matter of fact, this
is the part of the rationale behind the particular choice of the operator
family $\left(H_{\tau}\right)_{\tau\in\left[0,\pi\right]}$. We may
also describe this operator family via its quadratic form (the connection
between vertex conditions and quadratic forms for quantum graphs is
standard, see e.g. \cite{BerKuc_graphs}):
\begin{align}
Q_{\tau}\left[f\right]= & \int_{\Gamma_{\omega}(t)}\left|f'\right|^{2}dx~+\sum_{u\in s_{\omega}(t)\cup o(\Gamma_{\omega})}\alpha(u)\left|f(u)\right|^{2}\nonumber \\
 & +m_{a}(E)|f(v_{a}^{\mathrm{hor}})|^{2}-m_{a}(E)|f(v_{a}^{\mathrm{dec}})|^{2}\nonumber \\
 & +\cot\left(\tau/2\right)|f(v_{a}^{\mathrm{hor}})-f(v_{a}^{\mathrm{dec}})|^{2},\label{eq:quadratic-1}
\end{align}
where the Robin parameter $m_{a}\left(E\right)$ is fixed as above
and the domain of $Q_{\tau}$ is taken as all functions in $H^{1}\left(\Gamma_{\omega}\left(t\right)\right)$
which are continuous on $\Gamma_{a}$ and continuous on $\Gamma_{\omega}\left(t\right)\backslash\Gamma_{a}$
(but without requiring continuity at $v_{a}$, i.e., that $f(v_{a}^{\mathrm{hor}})=f(v_{a}^{\mathrm{dec}})$).
For $\tau=0$, the domain further restricts to functions satisfying
$f(v_{a}^{\mathrm{hor}})=f(v_{a}^{\mathrm{dec}})$ as well. Thus at
$\tau=0$ the operator satisfies also Kirchhoff conditions at $v_{a}$
(without splitting it into $v_{a}^{\mathrm{dec}}$ and $v_{a}^{\mathrm{hor}}$).

Thus the family $\left(H_{\tau}\right)_{\tau\in\left[0,\pi\right]}$
continuously interpolates between an operator on the full graph $\Gamma_{\omega}\left(t\right)$
(at $\tau=0$) and an operator on the cut graph $\Gamma_{\omega}\left(t\right)\backslash\Gamma_{a}$
(at $\tau=\pi$). Now, we consider all the other decorations (in addition
to $\Gamma_{a}$ discussed above) which are connected to $\Gamma_{\omega}(t)$.
We redefine the operator family $\left(H_{\tau}\right)_{\tau\in\left[0,\pi\right]}$
such that the vertex conditions at all vertices where the decorations
are attached are changed simultaneously in the same manner as for
$v_{a}$. Namely, the vertex conditions (\ref{eq: operator-family-VC-2-1}),
(\ref{eq: operator-family-VC-3-1}) are imposed at all these decoration
attachment vertices. The effect of this redefined operator family
$\left(H_{\tau}\right)_{\tau\in\left[0,\pi\right]}$ is equivalent
to disconnecting $\Gamma_{\omega}\left(t\right)$ at all these vertices
simultaneously at $\tau=\pi$. At $\tau=0$, we get the operator $\left.H_{\omega}\right|_{\Gamma_{\omega}(t)}$,
namely, the vertex conditions at the vertices of $\Gamma_{\omega}\left(t\right)$
are all Kirchhoff, except for the boundary vertices $s_{\omega}(t)\cup o(\Gamma_{\omega})$,
where the Robin conditions (\ref{eq: operator-family-VC-1-1}) are
imposed. What is important to emphasize is that exactly as above $\left.f_{\omega,E}\right|_{\Gamma_{\omega}(t)}$
is an eigenfunction of $H_{\tau}$ for all $\tau\in\left[0,\pi\right]$.

By standard methods (cf. \cite[thm. 1.4.4]{BerKuc_graphs}), $H_{\tau}$
are all self-adjoint with compact resolvents. Noticing that the map
$\tau\mapsto Q_{\tau}\left[f\right]$ is piecewise analytic with non-positive
derivative for all fixed $f$, a standard Kato-type \cite{Kato_book}
argument (see e.g. \cite{Sofer2022,BanProSof_ijm26} for detailed
proofs in similar systems) can be used to prove the following:
\begin{lem}
\label{lem:analytic-1}The eigenvalue branches $\left(\lambda_{n}\left(\tau\right)\right)_{n\in\N}$
of $\left(H_{\tau}\right)_{\tau\in\left[0,\pi\right]}$ are piecewise
real-analytic, and monotone non-increasing in any interval where they
are differentiable.
\end{lem}

We have collected all that is needed.
\begin{proof}[Proof of Lemma \ref{lem:Counting-lemma}]
The lemma considers the operators $\left.H_{\omega}\right|_{\Gamma_{\omega}(t)}$,
$\left.H_{\omega}\right|_{[0,tL]}$ and $\left.H_{\omega}\right|_{\Gamma_{a}}$.
Their spectral counting functions are denoted by $\sc$ , $\sch$
and $\sca$, respectively. We note that the operator $\left.H_{\omega}\right|_{\Gamma_{\omega}(t)}$
is exactly $H_{0}$ of the operator family $\left(H_{\tau}\right)_{\tau\in\left[0,\pi\right]}$
defined above. Denoting the spectral counting functions of this family
by $\sc^{(\tau)}(E)$ (i.e., $\sc^{(0)}=\sc$), the statement of Lemma
\ref{lem:Counting-lemma} translates to
\begin{equation}
\sc^{(0)}\left(E\right)=\sch\left(E\right)+\sum_{a\in\mathcal{A}}\ct t\left(\sca\left(E\right)-1\right).\label{eq:nwt}
\end{equation}
Now, consider the operator $H_{\pi}$. It is an operator on the cut
version of $\Gamma_{\omega}(t)$, namely the disjoint union of the
horizontal graph with the individual decorations. As such, its spectral
counting function equals the sum of the individual counting functions,
\begin{equation}
\sc^{(\pi)}\left(E\right)=\sch\left(E\right)+\sum_{a\in\mathcal{A}}\ct t\sca\left(E\right).\label{eq: lem-countin - spectral pi}
\end{equation}
 Given (\ref{eq: lem-countin - spectral pi}), we can prove (\ref{eq:nwt})
by showing the following properties on the eigenvalue curves:
\begin{enumerate}
\item \label{enu: lem-counting - prop-1} The operator family $\left(H_{\tau}\right)_{\tau\in[0,\pi]}$
is uniformly bounded from below.
\item \label{enu: lem-counting - prop-2} For $\varepsilon>0$ small enough,
there are exactly $\left\lfloor t\right\rfloor $ crossings of the
eigenvalue curves with the horizontal line $\lambda=E+\varepsilon$
in the interval $\tau\in[0,\pi]$.
\end{enumerate}
We first explain why the two properties above together with (\ref{eq: lem-countin - spectral pi})
imply (\ref{eq:nwt}), and then prove these properties. Consider the
rectangle bounded by $\tau=0$, $\tau=\pi$, $\lambda=E+\varepsilon$
and $\lambda=-C$, where $-C$ is a uniform lower bound of the family
$\left(H_{\tau}\right)_{\tau\in[0,\pi]}$ (see Figure~\ref{fig: SCurves-1}).
It is clear that $\sc^{(0)}\left(E\right)$ and $\sc^{(\pi)}\left(E\right)$
equal the number of intersections of the eigenvalue curves with the
left and right sides of the rectangle respectively. Due to property
(\ref{enu: lem-counting - prop-1}) there are no intersections with
the lower side of the rectangle. By Lemma~\ref{lem:analytic-1} the
eigenvalue curves are monotone non-increasing, and hence the number
of intersections with the top side equals $\sc^{(\pi)}\left(E\right)-\sc^{(0)}\left(E\right)$.
By property (\ref{enu: lem-counting - prop-2}) the number of these
intersections is $\left\lfloor t\right\rfloor $, so that 
\begin{equation}
\sc^{(\pi)}\left(E\right)-\sc^{(0)}\left(E\right)=\left\lfloor t\right\rfloor =\sum_{a\in\mathcal{A}}\ct t,\label{eq:spectral-shift}
\end{equation}
which combined with (\ref{eq: lem-countin - spectral pi}) proves
(\ref{eq:nwt}).

\begin{figure}
\includegraphics[scale=0.45]{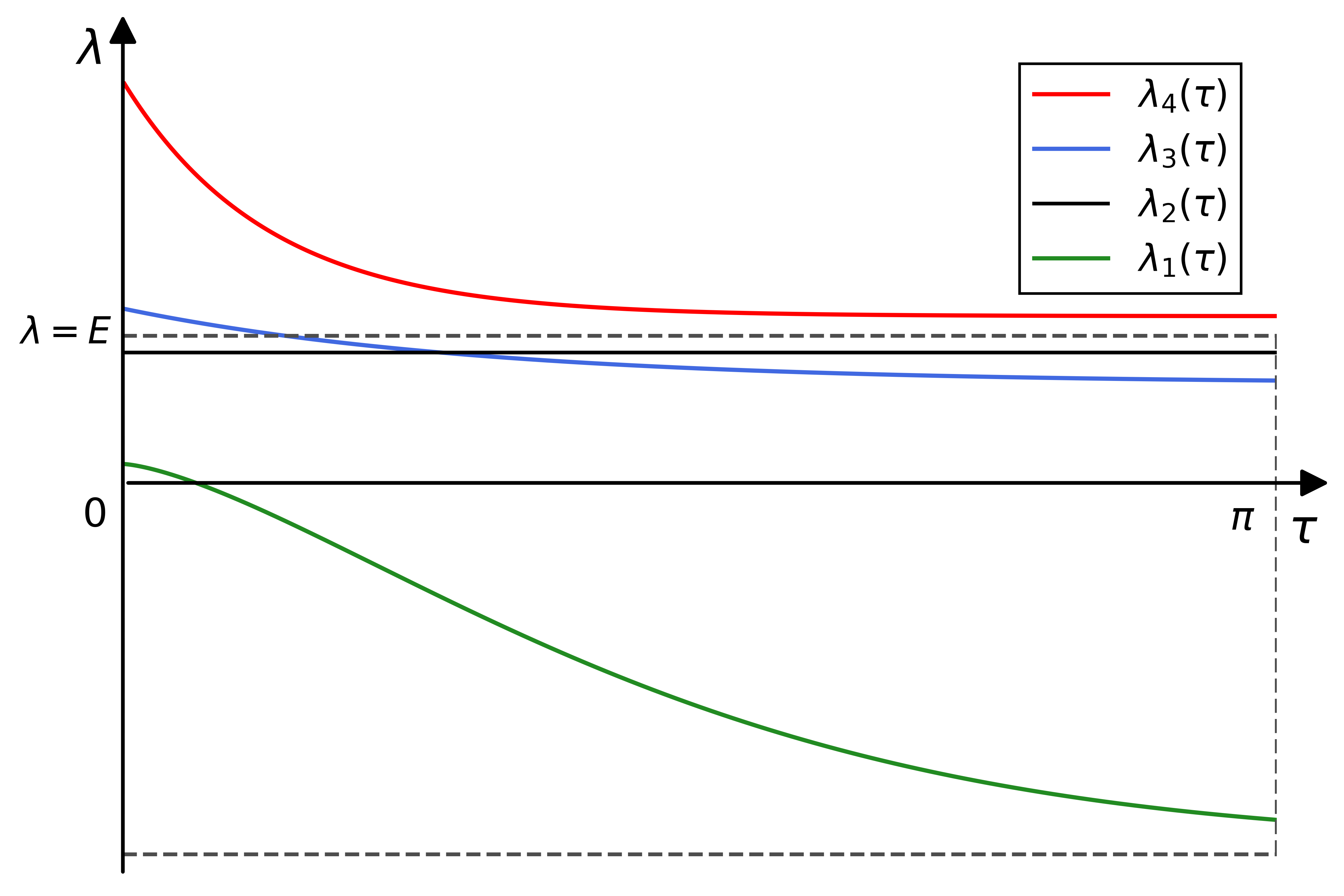}

\caption{Demonstration of the rectangle argument: the number of intersections
through the top side is equal to the spectral shift. \label{fig: SCurves-1}}
\end{figure}

\ 

Next, we prove the two properties mentioned above. Property (\ref{enu: lem-counting - prop-1})
follows from the quadratic form (\ref{eq:quadratic-1}): $\left(H_{\tau}\right)_{\tau\in\left[0,\pi\right]}$
is uniformly bounded from below since $\cot(\tau/2)$ is bounded from
below there.

For property (\ref{enu: lem-counting - prop-2}), first recall that
$E$ is as in the proof of Theorem \ref{thm:GLT} (and was used to
define the family $\left(H_{\tau}\right)_{\tau\in\left[0,\pi\right]}$).
We have already observed that the function $\fwe$ is an eigenfunction
of $H_{\tau}$ for all $\tau\in\left[0,\pi\right]$. Hence, there
exists a `flat' eigenvalue branch at the constant value of $\lambda=E$.
Denote by $M$ the multiplicity of $E$ as an eigenvalue of $H_{0}$
(we know that $M\geq1$, since $f_{\omega,E}$ is an eigenfunction).
 Next, we will show the following two statements: (a) The multiplicity
of $E$ as an eigenvalue of $H_{\pi}$ is $M+\left\lfloor t\right\rfloor $
and (b) The multiplicity of $E$ as an eigenvalue of $H_{\tau}$ for
any $\tau\in\left[0,\pi\right)$ is $M$. From this we would get that
there are exactly $\left\lfloor t\right\rfloor $ eigenvalue branch
which cross $\lambda=E$ and these crossings occur at $\tau=\pi$.
This immediately yields property (\ref{enu: lem-counting - prop-2})
which finishes the proof.

(a) Denote by $\{f^{(j)}\}_{j=1}^{M}$ a basis to the $E$-eigenspace
of $H_{0}$. We use the functions $\{f^{(j)}\}_{j=1}^{M}$ to construct
$M+\left\lfloor t\right\rfloor $ functions on $\Gamma_{\omega}(t)$.
Note that given the value $E$, at each of the $\left\lfloor t\right\rfloor $
decorations there is a unique (up to scalar) function which is a solution
of the ODE (\ref{eq:E-ODE}) with Kirchhoff conditions imposed at
all vertices, except at the base vertex where only a continuity condition
is imposed. The uniqueness is guaranteed since we demanded that for
all decorations $\gra$, $E$ is not in the spectrum of the decoration
with Dirichlet condition (if uniqueness is violated, one may construct
such a function which vanishes at the base vertex). We denote these
unique functions on the decorations by $\{f_{a}\}_{a\in\A}$. We thus
get that for each $f^{(j)}$ ($1\leq j\leq M$), its restriction to
each of the decorations either equals the corresponding $f_{a}$,
or identically vanishes at the decoration (in the case where $f^{(j)}(v_{a})=0$).

Given the above observations, we construct $M+\left\lfloor t\right\rfloor $
functions on $\Gamma_{\omega}(t)$ as follows. For each of the $\left\lfloor t\right\rfloor $
decorations construct a function which is supported only at this decoration
and vanishes everywhere else (i.e., it vanishes at all the other decorations
and at the horizontal line). This gives $\left\lfloor t\right\rfloor $
eigenfunctions of $H_{\pi}$. We additionally take $\{\left.f^{(j)}\right|_{\left[0,tL\right]}\}{}_{j=1}^{M}$,
which are also eigenfunctions of $H_{\pi}$. We thus get $M+\left\lfloor t\right\rfloor $
$E$-eigenfunctions of $H_{\pi}$. Clearly these functions are linearly
independent and we now show that there are no other $E$-eigenfunctions
of $H_{\pi}$. Assume by contradiction that there is another $E$-eigenfunction
of $H_{\pi}$, denoted by $g$ which is not a linear combination of
the $M+\left\lfloor t\right\rfloor $ functions mentioned above. To
get a contradiction, we construct a function $h$ on $\Gamma_{\omega}(t)$,
such that $\left.h\right|_{\left[0,tL\right]}=\left.g\right|_{\left[0,tL\right]}$
and at each decoration $\gra$ which is included in $\Gamma_{\omega}(t)$
we set $\left.h\right|_{\gra}$ to equal $f_{a}$ up to a scalar multiple
which is chosen to guarantee that $h$ is continuous at the base vertex
of $\gra$. This way, we obtain that $h$ is an $E$-eigenfunction
of $H_{0}$. At every decoration $\Gamma_{a}$, either $\left.g\right|_{\Gamma_{a}}$
equals $f_{a}$ up to a scalar multiple, or $\left.g\right|_{\Gamma_{a}}\equiv0$
(by the uniqueness mentioned above) . Therefore, $g$ is a linear
combination of $h$ and the $\left\lfloor t\right\rfloor $ functions
which are supported solely at the decorations of $\Gamma_{\omega}(t)$,
but this is a contradiction.

(b) Let $\tau\neq\pi$. Note that $\{f^{(j)}\}_{j=1}^{M}$ are $E$-eigenfunctions
of $H_{\tau}$. We should only show that there are no additional eigenfunctions.
Assume by contradiction that there is an eigenfunction $g$ of $H_{\tau}$
which is linearly independent of $\{f^{(j)}\}_{j=1}^{M}$. In particular,
on every decoration $\gra$ which is included in $\Gamma_{\omega}(t)$
we have that $g$ is a solution of the same ODE as $\left.f_{\omega,E}\right|_{\gra}$
(and as all of $\{\left.f^{(j)}\right|_{\gra}\}_{j=1}^{M}$). This
implies that at the base vertex $v_{a}$ of the decoration we get
$\sum_{e\sim v_{a}^{\mathrm{dec}}}\left.g'\right|_{e}\left(v_{a}^{\mathrm{dec}}\right)=m_{a}\left(E\right)g\left(v_{a}^{\mathrm{dec}}\right)$.
Comparing this to the vertex condition (\ref{eq: operator-family-VC-2-1})
and using $\cot(\tau/2)\neq0$ we get that $g(v_{a}^{\mathrm{dec}})=g(v_{a}^{\mathrm{hor}})$,
i.e., that $g$ is continuous at the base vertex of the decoration.
Since this is valid for all the decorations contained in $\Gamma_{\omega}(t)$
we get that $g$ is an $E$-eigenfunction of $H_{0}$ which is linearly
independent of $\{f^{(j)}\}_{j=1}^{M}$. A contradiction.
\end{proof}
\begin{rem*}
The proof of Lemma~\ref{lem:Counting-lemma} is based on the notion
of spectral flow. The spectral flow of an operator family such as
$\left(H_{\tau}\right)_{\tau\in[0,\pi]}$ is informally defined as
the number of oriented intersections of the eigenvalue curves of $H_{\tau}$
with some horizontal line $E=\mathrm{const}$. The spectral flow is
a topological invariant and an interesting framework in its own right.
To keep the proof self-contained, we instead used a direct argument.
We refer to \cite{BoossBavnbek2005,B.BoossBavnbek2013,B.BoossBavnbek2018}
and references therein for a thorough background about the spectral
flow, and also to \cite{BanProSof_ijm26,Prokhorova_prep,LatSuk_ams20}
for applications of the spectral flow specifically in the context
of quantum graphs.
\end{rem*}
\bibliographystyle{amsalpha}
\bibliography{GlobalBib_260311}

@Article{BanBerSmi_ahp12,
  author  = {Band, R. and Berkolaiko, G. and Smilansky, U.},
  title   = {Dynamics of Nodal Points and the Nodal Count on a Family of Quantum Graphs},
  number  = {1},
  pages   = {145--184},
  volume  = {13},
  journal = {Annales Henri Poincare},
  year    = {2012},
}

@Article{Bel_laa85,
  author  = {von Below, J.},
  title   = {A characteristic equation associated to an eigenvalue problem on {$c^2$}-networks},
  doi     = {10.1016/0024-3795(85)90258-7},
  pages   = {309--325},
  volume  = {71},
  journal = {Linear Algebra Appl.},
  mrclass = {94C15 (05C50)},
  year    = {1985},
}

@Article{Ber_cmp08,
  author   = {Berkolaiko, G.},
  journal  = {Comm. M. Phys.},
  title    = {A lower bound for nodal count on discrete and metric graphs},
  year     = {2008},
  number   = {3},
  pages    = {803--819},
  volume   = {278},
  doi      = {10.1007/s00220-007-0391-3},
  fjournal = {Communications in Mathematical Physics},
}

@Book{BerKuc_graphs,
  author    = {Berkolaiko, G. and Kuchment, P.},
  title     = {Introduction to Quantum Graphs},
  publisher = {AMS},
  series    = {Math. Surv. and Mon.},
  volume    = {186},
  year      = {2013},
}

@Article{Cat_mm97,
  author   = {Cattaneo, C.},
  journal  = {Monatsh. Math.},
  title    = {The spectrum of the continuous {L}aplacian on a graph},
  year     = {1997},
  number   = {3},
  pages    = {215--235},
  volume   = {124},
  doi      = {10.1007/BF01298245},
  fjournal = {Monatshefte f{\"u}r Mathematik},
  mrclass  = {35P05 (05C50 35J25 47A10 47F05)},
  mrnumber = {MR1476363 (98j:35127)},
  url      = {http://dx.doi.org/10.1007/BF01298245},
}

@Article{GnuSmiWeb_wrm04,
  author  = {Gnutzmann, S. and Smilansky, U. and Weber, J.},
  title   = {Nodal counting on quantum graphs},
  doi     = {10.1088/0959-7174/14/1/011},
  number  = {1},
  pages   = {S61--S73},
  volume  = {14},
  journal = {Waves Random Media},
  mrclass = {81Q05},
  year    = {2004},
}

@InCollection{GruLenVes_incol08,
  author    = {Gruber, M. J. and Lenz, D. H. and Veseli{\'c}, I.},
  booktitle = {Analysis on graphs and its applications},
  publisher = {Amer. Math. Soc.},
  title     = {Uniform existence of the integrated density of states for combinatorial and metric graphs over {${\mathbb{Z}}^d$}},
  year      = {2008},
  address   = {Providence, RI},
  pages     = {87--108},
  series    = {Proc. Sympos. Pure Math.},
  volume    = {77},
  mrclass   = {47B80 (47N50 60H25 81Q10 82B44)},
  mrnumber  = {MR2459865},
}

@Article{GruLenVes_jfa07,
  author   = {Gruber, M. J. and Lenz, D. H. and Veseli{\'c}, I.},
  journal  = {J. Funct. Anal.},
  title    = {Uniform existence of the integrated density of states for random {S}chr{\"o}dinger operators on metric graphs over {${\mathbb Z}^d$}},
  year     = {2007},
  number   = {2},
  pages    = {515--533},
  volume   = {253},
  doi      = {10.1016/j.jfa.2007.09.003},
  fjournal = {Journal of Functional Analysis},
  mrclass  = {82B44 (34L40 35J10 35P20 35R60 47B80 60H25 81Q10)},
  mrnumber = {MR2370087 (2009h:82048)},
  url      = {http://dx.doi.org/10.1016/j.jfa.2007.09.003},
}

@Article{LlePos_jmaa08,
  author   = {Lled{\'o}, F. and Post, O.},
  journal  = {J. Math. Anal. Appl.},
  title    = {Eigenvalue bracketing for discrete and metric graphs},
  year     = {2008},
  number   = {2},
  pages    = {806--833},
  volume   = {348},
  doi      = {10.1016/j.jmaa.2008.07.029},
  fjournal = {Journal of Mathematical Analysis and Applications},
  mrclass  = {47B39 (05C50 05C70 35P15 47A10)},
  mrnumber = {MR2446037 (2010a:47076)},
  url      = {http://dx.doi.org/10.1016/j.jmaa.2008.07.029},
}

@Article{Pan_lmp06,
  author   = {Pankrashkin, K.},
  journal  = {Lett. Math. Phys.},
  title    = {Spectra of {S}chr{\"o}dinger operators on equilateral quantum graphs},
  year     = {2006},
  number   = {2},
  pages    = {139--154},
  volume   = {77},
  doi      = {10.1007/s11005-006-0088-0},
  fjournal = {Letters in Mathematical Physics. A Journal for the Rapid Dissemination of Short Contributions in the Field of Mathematical Physics},
  mrclass  = {81Q10 (05C90 47A10 47F05 47N50)},
  mrnumber = {MR2251302 (2007f:81089)},
  url      = {http://dx.doi.org/10.1007/s11005-006-0088-0},
}

@Article{Sch_wrcm06,
  author   = {Schapotschnikow, P.},
  title    = {Eigenvalue and nodal properties on quantum graph trees},
  doi      = {10.1080/17455030600702535},
  number   = {3},
  pages    = {167--178},
  volume   = {16},
  fjournal = {Waves in Random and Complex Media. Propagation, Scattering and Imaging},
  journal  = {Waves Random Complex Media},
  mrclass  = {34B45 (34C10 81Q05)},
  year     = {2006},
}

@Article{AloBanBer_cmp18,
  author   = {Alon, L. and Band, R. and Berkolaiko, G.},
  journal  = {Comm. Math. Phys.},
  title    = {Nodal statistics on quantum graphs},
  year     = {2018},
  issn     = {0010-3616,1432-0916},
  number   = {3},
  pages    = {909--948},
  volume   = {362},
  doi      = {10.1007/s00220-018-3111-2},
  fjournal = {Communications in Mathematical Physics},
  mrclass  = {05C90},
  mrnumber = {3845291},
  url      = {https://doi.org/10.1007/s00220-018-3111-2},
}

@Book{Kato_book,
  author    = {Kato, T.},
  title     = {Perturbation theory for linear operators},
  edition   = {Second},
  note      = {Grundlehren der Mathematischen Wissenschaften, Band 132},
  pages     = {xxi+619},
  publisher = {Springer-Verlag},
  address   = {Berlin},
  mrclass   = {47-XX},
  mrnumber  = {0407617 (53 \#11389)},
  year      = {1976},
}

@PhdThesis{Alon_PhDThesis,
  author = {L. Alon},
  school = {Mathematics Department, Technion - Israel Institute of Technology},
  title  = {Quantum graphs - Generic eigenfunctions and their nodal count and Neumann count statistics},
  year   = {2020},
}

@InCollection{Berkolaiko_qg-intro17,
  author     = {Berkolaiko, G.},
  booktitle  = {Geometric and computational spectral theory},
  publisher  = {Amer. Math. Soc., Providence, RI},
  title      = {An elementary introduction to quantum graphs},
  year       = {2017},
  pages      = {41--72},
  series     = {Contemp. Math.},
  volume     = {700},
  doi        = {10.1090/conm/700/14182},
  mrclass    = {35R02 (35B05 35J10 81Q35)},
  mrnumber   = {3748521},
  mrreviewer = {Stefan Le Coz},
  url        = {https://doi.org/10.1090/conm/700/14182},
}

@InCollection{LatSuk_ams20,
  author     = {Latushkin, Y. and Sukhtaiev, S.},
  booktitle  = {Analytic trends in mathematical physics},
  publisher  = {Amer. Math. Soc., [Providence], RI},
  title      = {An index theorem for {S}chr\"{o}dinger operators on metric graphs},
  year       = {2020},
  pages      = {105--119},
  series     = {Contemp. Math.},
  volume     = {741},
  doi        = {10.1090/conm/741/14922},
  mrclass    = {34B45 (34B24 34L05 53D12 58J30 81Q35)},
  mrnumber   = {4047783},
  mrreviewer = {Jia Zhao},
  url        = {https://doi.org/10.1090/conm/741/14922},
}

@Article{GnuSmi_ap06,
  author   = {Gnutzmann, S. and Smilansky, U.},
  title    = {Quantum graphs: Applications to quantum chaos and universal spectral statistics},
  doi      = {10.1080/00018730600908042},
  number   = {5--6},
  pages    = {527--625},
  volume   = {55},
  fjournal = {Advances in Physics},
  journal  = {Adv. Phys.},
  year     = {2006},
}

@Article{B.BoossBavnbek2018,
  author  = {B. Booss-Bavnbek and C. Zhu},
  journal = {Memoirs of the American Mathematical Society},
  title   = {The {M}aslov Index in Symplectic {B}anach Spaces},
  year    = {2018},
}

@Article{B.BoossBavnbek2013,
  author  = {B. Booss-Bavnbek and C. Zhu},
  journal = {Ann Glob Anal Geom},
  title   = {The {M}aslov index in weak symplectic functional analysis},
  year    = {2013},
}

@Article{BanProSof_ijm26,
  author  = {R. Band and M. Prokhorova and G. Sofer},
  journal = {arXiv:2505.02039},
  title   = {Spectral flow and {R}obin domains on metric graphs},
}

@MastersThesis{Sofer2022,
  author = {G. Sofer},
  school = {Technion - Israel Institute of Technology},
  title  = {Spectral curves of quantum graphs with $\delta$s type vertex conditions},
  year   = {2022},
}

@Article{Lenz2009,
  author  = {D. Lenz and N. Peyerimhoff and O. Post and I. Veselic},
  journal = {Mathematical Physics, Analysis and Geometry 12},
  title   = {Continuity of the integrated density of states on random length metric graphs},
  year    = {2009},
}

@Article{Peyerimhoff2021,
  author  = {N. Peyerimhoff and M. {T\"{a}ufer}},
  journal = {Journal of Spectral Theory 11},
  title   = {Eigenfunctions and the integrated density of states on Archimedean tilings},
  year    = {2021},
}

@Book{Teschl2014,
  author    = {G. Teschl},
  publisher = {American Mathematical Society, Providence, Rhode Island},
  title     = {Mathematical methods in quantum mechanics with applications to {Schr\"{o}dinger} operators, second edition},
  year      = {2014},
}

@Article{Damanik2023,
  author  = {D. Damanik and J. Fillman and Z. Zhang},
  journal = {Journal of Spectral Theory 13},
  title   = {Johnson {S}chwartzman gap labelling for ergodic {J}acobi matrices},
  year    = {2023},
}

@Article{AloBanBer_exp22,
  author    = {Alon, L. and Band, R. and Berkolaiko, G.},
  journal   = {Experimental Mathematics},
  title     = {Universality of Nodal Count Distribution in Large Metric Graphs},
  year      = {2022},
  number    = {0},
  pages     = {1-35},
  volume    = {0},
  doi       = {10.1080/10586458.2022.2092565},
  eprint    = {https://doi.org/10.1080/10586458.2022.2092565},
  publisher = {Taylor & Francis},
  url       = {https://doi.org/10.1080/10586458.2022.2092565},
}

@Article{DamEmbFilMei_exp23,
  author    = {D. Damanik and M. Embree and J. Fillman and M. Mei},
  journal   = {Experimental Mathematics},
  title     = {Discontinuities of the Integrated Density of States for {L}aplacians Associated with {P}enrose and {A}mmann--{B}eenker Tilings},
  year      = {2023},
  number    = {0},
  pages     = {1-23},
  volume    = {0},
  doi       = {10.1080/10586458.2023.2206589},
  eprint    = {https://doi.org/10.1080/10586458.2023.2206589},
  publisher = {Taylor & Francis},
  url       = {https://doi.org/10.1080/10586458.2023.2206589},
}

@Article{BaaGahMaz_ijc24,
  author       = {M. Baake and F. G\"{a}hler and J. Maz\'{a}\v{c}},
  journal      = {Israel Journal of Chemistry},
  title        = {On the {F}ibonacci tiling and its modern ramifications},
  year         = {2024},
  eprint       = {2311.05387},
  primaryclass = {math.MG},
}

@Book{Baake2013,
  author     = {Baake, M. and Grimm, U.},
  publisher  = {Cambridge University Press, Cambridge},
  title      = {Aperiodic order. {V}ol. 1},
  year       = {2013},
  isbn       = {978-0-521-86991-1},
  note       = {A mathematical invitation, With a foreword by Roger Penrose},
  series     = {Encyclopedia of Mathematics and its Applications},
  volume     = {149},
  doi        = {10.1017/CBO9781139025256},
  mrclass    = {52C23 (11H06 20C35 20H15 82D25)},
  mrnumber   = {3136260},
  mrreviewer = {Jean-Pierre Gazeau},
  pages      = {xvi+531},
  url        = {http://dx.doi.org/10.1017/CBO9781139025256},
}

@Book{Queffelec_book10,
  author     = {Queff\'{e}lec, M.},
  publisher  = {Springer-Verlag, Berlin},
  title      = {Substitution dynamical systems---spectral analysis},
  year       = {2010},
  edition    = {Second},
  isbn       = {978-3-642-11211-9},
  series     = {Lecture Notes in Mathematics},
  volume     = {1294},
  doi        = {10.1007/978-3-642-11212-6},
  mrclass    = {37B10 (11K06 28D05 37A30 37A45 37B15)},
  mrnumber   = {2590264},
  mrreviewer = {T\'{u}lio\ O.\ Carvalho},
  pages      = {xvi+351},
  url        = {https://doi.org/10.1007/978-3-642-11212-6},
}

@Article{Bellissard1992,
  author     = {Bellissard, J. and Bovier, A. and Ghez, J.-M.},
  journal    = {Rev. Math. Phys.},
  title      = {Gap labelling theorems for one-dimensional discrete {S}chr\"odinger operators},
  year       = {1992},
  issn       = {0129-055X},
  number     = {1},
  pages      = {1--37},
  volume     = {4},
  fjournal   = {Reviews in Mathematical Physics. A Journal for Both Review and Original Research Papers in the Field of Mathematical Physics},
  mrclass    = {47N50 (19K33 39A12 46L99 47B39 81Q10)},
  mrnumber   = {1160136},
  mrreviewer = {G. V. Rozenblum},
  url        = {https://doi.org/10.1142/S0129055X92000029},
}

@Book{damFil_book22,
  author    = {Damanik, D. and Fillman, J.},
  publisher = {American Mathematical Society},
  title     = {One-dimensional Ergodic {S}chr{\"o}dinger Operators: I. General Theory},
  year      = {2022},
  volume    = {221},
}

@Book{Kurasov_book24,
  author    = {Kurasov, P.},
  publisher = {Birkh\"auser/Springer, Berlin},
  title     = {Spectral geometry of graphs},
  year      = {[2024] \copyright 2024},
  isbn      = {978-3-662-67870-1; 978-3-662-67872-5},
  series    = {Operator Theory: Advances and Applications},
  volume    = {293},
  doi       = {10.1007/978-3-662-67872-5},
  mrclass   = {81-02 (05-02 34B45 47A10 81Q35 81U40)},
  mrnumber  = {4697523},
  pages     = {xvi+639},
  url       = {https://doi.org/10.1007/978-3-662-67872-5},
}

@Article{SchFraFliGru_prl24,
  author    = {Schirmann, J. and Franca, S. and Flicker, F. and Grushin, A. G.},
  journal   = {Phys. Rev. Lett.},
  title     = {Physical Properties of an Aperiodic Monotile with Graphene-like Features, Chirality, and Zero Modes},
  year      = {2024},
  month     = {Feb},
  pages     = {086402},
  volume    = {132},
  doi       = {10.1103/PhysRevLett.132.086402},
  issue     = {8},
  numpages  = {8},
  publisher = {American Physical Society},
  url       = {https://link.aps.org/doi/10.1103/PhysRevLett.132.086402},
}

@Misc{BanSof_prep24a,
  author = {R. Band and G. Sofer},
  title  = {Zero measure {C}antor spectrum for discrete and metric aperiodic tiling graphs},
  year   = {In preparation},
}

@Article{Simon1982,
  author     = {Simon, B.},
  journal    = {Adv. in Appl. Math.},
  title      = {Almost periodic {S}chr\"odinger operators: a review},
  year       = {1982},
  issn       = {0196-8858},
  number     = {4},
  pages      = {463--490},
  volume     = {3},
  doi        = {10.1016/S0196-8858(82)80018-3},
  fjournal   = {Advances in Applied Mathematics},
  mrclass    = {34B25 (47E05 58F11 81-02 81C05)},
  mrnumber   = {682631},
  mrreviewer = {Evans M. Harrell},
  url        = {http://dx.doi.org/10.1016/S0196-8858(82)80018-3},
}

@Misc{DamFill_book_vol_2,
  author = {Damanik, D. and Fillman, J.},
  note   = {In preparation},
  title  = {One-dimensional ergodic {S}chr\"{o}dinger operators---{II}. {S}pecific Classes},
}

@Article{JohMos_cmp82,
  author     = {Johnson, R. and Moser, J.},
  journal    = {Comm. Math. Phys.},
  title      = {The rotation number for almost periodic potentials},
  year       = {1982},
  issn       = {0010-3616,1432-0916},
  number     = {3},
  pages      = {403--438},
  volume     = {84},
  fjournal   = {Communications in Mathematical Physics},
  mrclass    = {34B25 (58F07 58F19)},
  mrnumber   = {667409},
  mrreviewer = {H.\ Hochstadt},
  url        = {http://projecteuclid.org/euclid.cmp/1103921211},
}

@Article{Johnson_jde86,
  author     = {Johnson, R. A.},
  journal    = {J. Differential Equations},
  title      = {Exponential dichotomy, rotation number, and linear differential operators with bounded coefficients},
  year       = {1986},
  issn       = {0022-0396,1090-2732},
  number     = {1},
  pages      = {54--78},
  volume     = {61},
  doi        = {10.1016/0022-0396(86)90125-7},
  fjournal   = {Journal of Differential Equations},
  mrclass    = {47E05 (34A20 47B39 58F19)},
  mrnumber   = {818861},
  mrreviewer = {A.\ Pelczar},
  url        = {https://doi.org/10.1016/0022-0396(86)90125-7},
}

@Article{Schwartzman_anmath57,
  author     = {Schwartzman, S.},
  journal    = {Ann. of Math. (2)},
  title      = {Asymptotic cycles},
  year       = {1957},
  issn       = {0003-486X},
  pages      = {270--284},
  volume     = {66},
  doi        = {10.2307/1969999},
  fjournal   = {Annals of Mathematics. Second Series},
  mrclass    = {54.0X},
  mrnumber   = {88720},
  mrreviewer = {J.\ C.\ Oxtoby},
  url        = {https://doi.org/10.2307/1969999},
}

@InCollection{DamFil_otaa23,
  author    = {Damanik, D. and Fillman, J.},
  booktitle = {From complex analysis to operator theory---a panorama},
  publisher = {Birkh\"auser/Springer, Cham},
  title     = {Gap labelling for discrete one-dimensional ergodic {S}chr\"odinger operators},
  year      = {[2023] \copyright 2023},
  isbn      = {978-3-031-31138-3; 978-3-031-31139-0},
  pages     = {341--404},
  series    = {Oper. Theory Adv. Appl.},
  volume    = {291},
  doi       = {10.1007/978-3-031-31139-0\_14},
  mrclass   = {47B36 (37A20 47A10 81Q10)},
  mrnumber  = {4651279},
  url       = {https://doi.org/10.1007/978-3-031-31139-0_14},
}

@Article{AvrOsaSei_phystod03,
  author  = {Avron, J. E. and Osadchy, D. and Seiler, R.},
  journal = {Physics Today},
  title   = {A Topological Look at the {Q}uantum {H}all {E}ffect},
  year    = {2003},
  issn    = {0031-9228},
  month   = {08},
  number  = {8},
  pages   = {38-42},
  volume  = {56},
  doi     = {10.1063/1.1611351},
  url     = {https://doi.org/10.1063/1.1611351},
}

@Article{BelLimTes_cmp83,
  author     = {J. Bellissard and R. Lima and D. Testard},
  journal    = {Comm. Math. Phys.},
  title      = {A metal-insulator transition for the almost {M}athieu model},
  year       = {1983},
  issn       = {0010-3616,1432-0916},
  number     = {2},
  pages      = {207--234},
  volume     = {88},
  fjournal   = {Communications in Mathematical Physics},
  mrclass    = {82A57 (39B10 47B39 81C10 82A65)},
  mrnumber   = {696805},
  mrreviewer = {W.\ Kirsch},
  url        = {http://projecteuclid.org/euclid.cmp/1103922281},
}

@Article{DamEmiFil_jst23,
  author   = {D. Damanik and I. Emilsd\'ottir and J. Fillman},
  journal  = {J. Spectr. Theory},
  title    = {The {S}chwartzman group of an affine transformation},
  year     = {2023},
  issn     = {1664-039X,1664-0403},
  number   = {4},
  pages    = {1281--1296},
  volume   = {13},
  doi      = {10.4171/jst/476},
  fjournal = {Journal of Spectral Theory},
  mrclass  = {35J10 (47A10 47B36 58J51)},
  mrnumber = {4707545},
  url      = {https://doi.org/10.4171/jst/476},
}

@Book{Lothaire2002,
  author     = {M. Lothaire},
  publisher  = {Cambridge University Press},
  title      = {Algebraic Combinatorics on Words},
  year       = {2002},
  series     = {Encyclopedia of Mathematics and its Applications},
  collection = {Encyclopedia of Mathematics and its Applications},
  place      = {Cambridge},
}

@Article{KlaLenSto_cmp03,
  author     = {S. Klassert and D. Lenz and P. Stollmann},
  journal    = {Comm. Math. Phys.},
  title      = {Discontinuities of the integrated density of states for random operators on {D}elone sets},
  year       = {2003},
  issn       = {0010-3616,1432-0916},
  number     = {2-3},
  pages      = {235--243},
  volume     = {241},
  doi        = {10.1007/s00220-003-0920-7},
  fjournal   = {Communications in Mathematical Physics},
  mrclass    = {82B44 (47B80 47N50 52C22)},
  mrnumber   = {2013799},
  mrreviewer = {David\ Damanik},
  url        = {https://doi.org/10.1007/s00220-003-0920-7},
}

@Article{LenMulVes_pos08,
  author     = {D. Lenz and P. M\"uller and I. Veseli\'c},
  journal    = {Positivity},
  title      = {Uniform existence of the integrated density of states for models on {$\Bbb Z^d$}},
  year       = {2008},
  issn       = {1385-1292,1572-9281},
  number     = {4},
  pages      = {571--589},
  volume     = {12},
  doi        = {10.1007/s11117-008-2238-3},
  fjournal   = {Positivity. An International Mathematics Journal Devoted to Theory and Applications of Positivity},
  mrclass    = {47B80 (37A30 37D35 47N50 81Q10)},
  mrnumber   = {2448749},
  mrreviewer = {Hatem\ Najar},
  url        = {https://doi.org/10.1007/s11117-008-2238-3},
}

@Article{Prokhorova_prep,
  author  = {M. Prokhorova},
  journal = {In preparation},
  title   = {Spectral flow in finite quantum graphs},
}

@Misc{Kellendonk_private_2026,
  author = {J. Kellendonk},
  title  = {private communication},
}

@Article{Oxtoby1952,
  author     = {Oxtoby, J. C.},
  journal    = {Bull. Amer. Math. Soc.},
  title      = {Ergodic sets},
  year       = {1952},
  issn       = {0002-9904},
  pages      = {116--136},
  volume     = {58},
  doi        = {10.1090/S0002-9904-1952-09580-X},
  fjournal   = {Bulletin of the American Mathematical Society},
  mrclass    = {46.3X},
  mrnumber   = {47262},
  mrreviewer = {K.\ Yosida},
  url        = {https://doi.org/10.1090/S0002-9904-1952-09580-X},
}

@Article{BoossBavnbek2005,
  author    = {B. Booss-Bavnbek and M. Lesch and J. Phillips},
  journal   = {Canadian Journal of Mathematics},
  title     = {Unbounded {F}redholm operators and spectral flow},
  year      = {2005},
  number    = {2},
  pages     = {225--250},
  volume    = {57},
  publisher = {Cambridge University Press},
}

@Article{DamLi_jfa2025,
  author   = {D. Damanik and L. Li},
  journal  = {Journal of Functional Analysis},
  title    = {Opening gaps in the spectrum of strictly ergodic {J}acobi and {CMV} matrices},
  year     = {2025},
  issn     = {0022-1236},
  number   = {12},
  pages    = {111182},
  volume   = {289},
  doi      = {https://doi.org/10.1016/j.jfa.2025.111182},
  url      = {https://www.sciencedirect.com/science/article/pii/S0022123625003647},
}

@Article{PeyTauVes_nano_2017,
  author  = {Peyerimhoff, N. and T{\"a}ufer, M. and Veseli{\'c}, I.},
  journal = {Nanosystems: Physics, Chemistry, Mathematics},
  title   = {Unique continuation principles and their absence for {Schr{\"o}dinger} eigenfunctions on combinatorial and quantum graphs and in continuum space},
  year    = {2017},
  number  = {2},
  pages   = {216--230},
  volume  = {8},
  doi     = {10.17586/2220-8054-2017-8-2-216-230},
}

@Article{Kellendonk2024,
  author  = {J. Kellendonk},
  journal = {Israel Journal of Chemistry},
  title   = {Topological Quantum Numbers in Quasicrystals},
  year    = {2024},
  issn    = {0021-2148},
  note    = {Overview on the theory of topological quantum numbers in quasicrystals via non-commutative topology},
  doi     = {10.1002/ijch.202400027},
}

@Article{BreLev_arXiv_2026,
  author       = {J. Breuer and N. Y. Levi},
  journal      = {arXiv:2603.08362},
  title        = {The Point Spectrum Of Periodic Quantum Trees},
  year         = {2026},
  primaryclass = {math.SP},
  url          = {https://arxiv.org/abs/2603.08362},
}

@Article{Levi_prep_2026,
  author  = {N. Y. Levi},
  journal = {In preparation},
  title   = {Heat Kernels and a Trace Formula on Quantum Graphs with a Potential},
  year    = {2026},
}

@Unpublished{Band,
  author = {R. Band and G. Sofer},
  note   = {In preparation},
  title  = {The {D}ry {T}en {M}artini {P}roblem for {S}turmian metric graphs},
}

\end{document}